\documentclass{article}
\usepackage{graphicx,amsmath,amssymb,amsthm,xcolor} 
\usepackage[left=2cm,right=2cm,top=2cm,bottom=2cm]{geometry}
\newtheorem{theorem}{Theorem}[section]  
\newtheorem{definition}{Definition}[section]
\newtheorem{lemma}{Lemma}[section]
\newtheorem{proposition}{Proposition}[section]
\newtheorem{cor}[theorem]{Corollary}
\newtheorem*{conjecture}{Conjecture}

\newtheorem{remark}{Remark}[section]
\title{On the Lipschitz continuity of the Spherical Cap Discrepancy around generic point sets}
\author{H. Heitsch \and R. Henrion }
\date{\today}

\begin{document}

\maketitle
\begin{abstract}
\noindent
The spherical cap discrepancy is a prominent measure of uniformity for sets on the d-dimensional sphere. It is particularly important for estimating the integration error for certain classes of functions on the sphere. Building on a recently proven explicit formula for the spherical discrepancy, we show as a main result of this paper that this discrepancy is Lipschitz continuous in a neighbourhood of so-called generic point sets (as they are typical outcomes of Monte-Carlo sampling). This property may have some impact (both algorithmically and theoretically for deriving necessary optimality conditions) on optimal quantization, i.e., on finding point sets of fixed size on the sphere having minimum spherical discrepancy.
\end{abstract}
{\bf Keywords:} Spherical cap discrepancy, Uniform distribution on sphere, Lipschitz continuity, Necessary optimality conditions.

\noindent
{\bf MSC:} 11K38, 49K30

\section{Introduction}

Point sets uniformly located on the classical or higher dimensional sphere are of much interest in many disciplines of mathematics. As examples we refer to point cloud interpolation in computer vision
\cite{Nguyen2023} or to optimization problems with chance constraints using the so-called spherical-radial decomposition of elliptically distributed (e.g., Gaussian) random vectors \cite{vanAckooij_Henrion_2014}. Uniformity of point sets on the sphere can be characterized by various criteria, e.g., the sum of pairwise distances (which should be large) or by its Coulomb energy (which should be small). If the focus is on estimating the integration error when replacing a spherical integral of a function by an average function value on the spherical point set, then the so-called {\it spherical cap discrepancy} is a natural measure of goodness for the uniformity of this point set \cite{Grabner1993,Aistleitner2012,brauchart2014}. Contrary to the criteria mentioned above, the spherical cap discrepancy (being defined as a supremum of infinitely many local discrepancies) is originally not endowed with an explicit formula which could be used for its numerical computation or for its minimization as a function of the point set. This did not harm theoretical investigations in the context of the construction of low discrepancy sequences but it became obstructive in numerical experiments. A possible remedy consisted in reducing the supremum to a maximum over finitely many local discrepancies (e.g. \cite[p.1005]{Aistleitner2012}), but, of course, this provides just a lower bound which might deviate considerably from the true value \cite[p.13]{heitsch-henrion}. A certainly more precise algorithmic approximation was provided in \cite{Bakhshizadeh2020}, but still it was not based on an exact formula and moreover restricted to the classical two-dimensional sphere. In \cite{heitsch-henrion}, a precise enumerative formula for the spherical cap discrepancy was derived, which reduced the supremum over an infinite family of local discrepancies to a finite maximum of fully explicit and numerically easy to compute expressions. Not surprisingly, this formula suffers from a poor complexity. Nonetheless, it could be used for calibration purposes for moderate sizes of the point set and small dimensions of the sphere (in \cite{heitsch-henrion}, sets with 2000 points in the two-dimensional sphere to 100 points in the five-dimensional sphere were considered). For a practical application of this formula in image analysis, we refer to \cite{Nguyen2023}.

It turns out that, apart from its numerical use, the mentioned formula maybe of interest in characterizing the spherical cap discrepancy as a function of the point set. This observation is based on the fact that the finitely many expressions whose maximum constitutes the spherical cap discrepancy are fully explicit functions of the point set. This allows us, beyond proving the continuity of the spherical cap discrepancy by elementary arguments, to verify even its Lipschitz continuity around so-called {\it generic} point sets. The latter refers to point sets on the sphere for which each selection of cardinality not larger than the space dimension is linearly independent. Such point sets are typical outcomes of Monte-Carlo (but not of Quasi Monte-Carlo) sampling. The main argument for proving Lipschitz continuity relies on the fact that, locally around a generic point set, the spherical cap discrepancy can be represented as a continuous selection of ${\mathcal{C}^1}$-functions (see \cite{Scholtes2012}). Moreover, we are able to provide explicitly computable Lipschitz constants. This might be of interest in the application of global optimization methods for minimizing the spherical cap discrepancy (optimal quantization) for a fixed sample size. Note that low discrepancy sequences whose design on the sphere is an active field of research have nice asymptotic properties but do not guarantee optimality for a fixed sample size. Apart from algorithmic relevance, the proven Lipschitz continuity paves a way for establishing necessary optimality conditions in optimal quantization on the sphere by means of the Clarke subdifferential \cite{clarke1983}.

The paper is organized as follows: Section 2 briefly introduces some basic concepts, presents some simple preliminary results needed later on and proves the continuity of the spherical cap discrepancy. In Section 3, a representation formula for the spherical cap discrepancy as a maximum of finitely many (explicit) functions around generic point sets is proven. In Section 4, an extended cap discrepancy is introduced and its Lipschitz continuity around generic point sets is verified. As a trivial consequence, the same property for the original discrepancy is derived as the main result of the paper. Section 5 briefly describes how the previous results could applied in order to derive necessary conditions for optimal quantization with respect to the spherical cap discrepancy.

\section{Basic concepts and continuity of the spherical cap discrepancy}

We start by defining the following family of subsets of $\mathbb{R}^d$:
\[
H(w,t):=\left\{ x\in \mathbb{R}^{d}|\left\langle w,x\right\rangle \geq
t\right\} \quad \left( w\in \mathbb{R}^{d},\,t\in \mathbb{R}
\right). 
\]
If $w\neq 0$, then $H(w,t)$ represents a closed half space in $\mathbb{R}^{d}$, otherwise it coincides with either $\mathbb{R}^{d}$ or the empty set depending on whether $t\leq 0$ or $t>0$. With each of these sets, we associate its so-called cap measure on the sphere:
\[
\mu ^{cap}\left( w,t\right):=\sigma \left( \mathbb{S}^{d-1}\cap
H(w,t)\right)\quad (\sigma =\mbox{ law of uniform distribution on } \mathbb{S}^{d-1}),
\]
where $\mathbb{S}^{d-1}$ refers to the ($d-1$)-dimensional Euclidean unit sphere in $\mathbb{R}^{d}$. We assume in the following that $d\geq 2$.

For a matrix $X=\left(x^{(1)},\ldots ,x^{(N)}\right)$ of order $(d,N)$ with $N\geq 1$ representing a set of points $\left\{x^{(1)},\ldots ,x^{(N)}\right\}\subseteq\mathbb{S}^{d-1}$, the empirical measure induced from this point set assigns to the set $\mathbb{S}^{d-1}\cap H(w,t)$ its empirical probability
\[
\mu ^{emp}\left( X,w,t\right):=N^{-1}\cdot \#\left\{ i\in \left\{
1,\ldots ,N\right\} |x^{(i)}\in\mathbb{S}^{d-1}\cap H(w,t)\right\}=N^{-1}\cdot \#\left\{ i\in \left\{
1,\ldots ,N\right\} |x^{(i)}\in H(w,t)\right\}.
\]
As a side remark we note that the following relation is immediate from the definition:
\begin{equation}\label{empcomp}
\mu ^{emp}\left( X,w,t\right) +\mu ^{emp}\left( X,-w,-t\right)=1+N^{-1}\cdot\#\left\{ i\in \left\{
1,\ldots ,N\right\} |\langle w,x^{(i)}\rangle=t\right\}.
\end{equation}
In order to measure the uniformity of a point set on the sphere, one might compare the deviation between its cap measure and empirical measure on all sets $\mathbb{S}^{d-1}\cap H(w,t)$:
\begin{equation}\label{initialdisc}
\Delta^0\left(X\right) := \sup_{w\in\mathbb{R}^{d},\,t\in\mathbb{R}}\left\vert \mu ^{emp}\left( X,w,t\right)
-\mu ^{cap}\left(w,t\right) \right\vert\quad \big(X\in\left(\mathbb{S}^{d-1}\right)^N\big).    
\end{equation}
Clearly, the smaller $\Delta^0$, the better both measures coincide on the chosen family of sets. Such quantities are called discrepancies. If one restricts the family of sets $H(w,t)$ to those with $(w,t)\in\mathbb{S}^{d-1}\times [-1,1]$, then one obtains the so-called {\it spherical cap discrepancy} (e.g., \cite{brauchart2014})
\begin{equation}\label{capdef0}
\Delta \left(X\right) := \sup_{w\in\mathbb{S}^{d-1},\,t\in \left[ -1,1\right]}\left\vert \mu ^{emp}\left( X,w,t\right)
-\mu ^{cap}\left(w,t\right) \right\vert\quad \big(X\in\left(\mathbb{S}^{d-1}\right)^N\big).
\end{equation}
Observe, that for $(w,t)\in\mathbb{S}^{d-1}\times [-1,1]$, the sets $H(w,t)$ represent closed half spaces with normal vector $w$ and height $t$. Their intersections $\mathbb{S}^{d-1}\cap H(w,t)$, on which the empirical measure and the uniform distribution are compared, are nonempty and called {\it spherical caps}.
Some authors define the spherical cap discrepancy by using open half spaces instead, i.e., by imposing the strict inequality $\langle w,x\rangle >t$ in the definition of $H(w,t)$ (e.g., \cite{Etayo2021}). One could formally refer to this alternative definition as a discrepancy $\Delta^1(X)$. It is easy to see that all these three discrepancy definitions coincide, i.e., $\Delta (X)=\Delta^0 (X)=\Delta^1 (X)$. We provide a proof in Proposition \ref{disccoinc} of the appendix for the reader's convenience. We shall base this paper on the representation \eqref{capdef0}, but occasionally, the equality with \eqref{initialdisc} may turn out to be useful.

If $w\in\mathbb{S}^{d-1}$, then the cap measure does not depend on $w$ and we simply write $\mu ^{cap}\left(t\right):=\mu ^{cap}\left( w,t\right)$. In this case, the following explicit formula is well known (e.g., \cite{li2011}):
\begin{equation}
\mu ^{cap}\left(t\right)=\left\{ 
\begin{array}{clc}
C_{d}\displaystyle\int_{0}^{\arccos (t)}\sin ^{d-2}(\tau )d\tau , & \quad %
\mbox{if} & 0\leq t\leq 1, \\ 
\displaystyle1-C_{d}\int_{0}^{\arccos (-t)}\sin ^{d-2}(\tau )d\tau , & \quad %
\mbox{if} & -1\leq t<0,%
\end{array}%
\right.  \label{capmeasure}
\end{equation}%
where 
\begin{equation}\label{cddef}
C_{d}:=\frac{1}{\int_{0}^{\pi }\sin ^{d-2}(\tau )d\tau }
\end{equation}%
is some normalizing constant. It follows immediately from \eqref{capmeasure} that $\mu ^{cap}$ is continuous and that
\begin{equation}\label{capcomp}
\mu ^{cap}(t)=1-\mu ^{cap}(-t)\quad\forall t\in [-1,1].    
\end{equation}
Therefore, we shall work from now on with the following form of \eqref{capdef0}:
\begin{equation}\label{capdef}
\Delta \left(X\right) := \sup_{w\in\mathbb{S}^{d-1},\,t\in \left[ -1,1\right]}\left\vert \mu ^{emp}\left( X,w,t\right)
-\mu ^{cap}\left(t\right) \right\vert\quad \big(X\in\left(\mathbb{S}^{d-1}\right)^N\big).
\end{equation}
We collect three properties of the spherical cap discrepancy
that are direct consequences of the definition \eqref{capdef}. We observe first, that the supremum in \eqref{capdef} is actually a maximum and that the spherical cap realizing this maximum contains at least one element of the given point set on its relative boundary:
\begin{proposition}[\cite{heitsch-henrion}, Proposition 1 \& 2]\label{maxcap}
Let $X\in\left(\mathbb{S}^{d-1}\right)^N$ be given. Then, there are $w^*\in\mathbb{S}^{d-1}$ and $t^*\in[-1,1]$
such that
\[
\Delta(X) = |\mu^{emp}(X,w^*,t^*) - \mu^{cap}(t^*)|.
\]
Moreover, there exists some $i\in\{1,\ldots ,N\}$ with $\langle w^*,x^{(i)}\rangle =t^*$.
\end{proposition}

\noindent
Secondly, we state a general lower bound for $\Delta(X)$
that depends on the space dimension and the number of points, but not on the position of the points
on the sphere.
\begin{proposition}\label{lowerbound2}
Let be $\kappa:=\min\{d,N\}$.
One has that $\Delta (X)\geq \kappa (2N)^{-1}>0$ for all $X\in\left(\mathbb{S}^{d-1}\right)^N$.
\end{proposition}
\begin{proof}
Chose some $w\in\mathbb{S}^{d-1}$ such that
$\langle w,x^{(1)}-x^{(j)}\rangle = 0$ for all $j=2,\ldots,\kappa$ and put 
$t:=\langle w, x^{(1)}\rangle$. Then, $|t| \leq 1$, and we have that $\langle w, x^{(i)}\rangle = t$
for $i=1,\ldots,\kappa$. Therefore, owing to \eqref{empcomp} and \eqref{capcomp},
\begin{eqnarray*}
2\Delta(X)&\geq& |\mu^{emp}(X,w,t) - \mu^{cap}(t)| + |\mu^{emp}(X,-w,-t) - \mu^{cap}(-t)|\\&\geq&
|\mu^{emp}(X,w,t) + \mu^{emp}(X,-w,-t) - \mu^{cap}(t) - \mu^{cap}(-t)|\\&=&
|1+N^{-1}\cdot\#\{ i\in \left\{
1,\ldots ,N\right\} |\langle w,x^{(i)}\rangle=t\}-1|
=N^{-1}\cdot\#\{ i\in \left\{
1,\ldots ,N\right\} |\langle w,x^{(i)}\rangle=t\}\geq N^{-1}\kappa .
\end{eqnarray*}
\end{proof}
\noindent
A further property we want to adapt from \cite{heitsch-henrion} is a slightly stronger version
of \cite[Corollary 1]{heitsch-henrion}. We observe that
the empirical measure is always strictly greater than the cap measure for any $(w^*,t^*)$
realizing the spherical cap discrepancy.
\begin{proposition}\label{empgreatercap}
For $(w^*,t^*)$ realizing $\Delta(X)$ in Proposition \ref{maxcap} it holds that
$\mu^{emp}(X,w^*,t^*) > \mu^{cap}(t^*)$.
\end{proposition}
\begin{proof}
By assumption, we have that $ \Delta(X) = |\mu^{emp}(X,w^*,t^*) - \mu^{cap}(t^*)|$.
From \cite[Corollary 1]{heitsch-henrion} we already know that
$\mu^{emp}(X,w^*,t^*) \geq \mu^{cap}(t^*)$. Now, the equality $\mu^{emp}(X,w^*,t^*) = \mu^{cap}(t^*)$
would imply $\Delta(X)=0$, a contradiction with Proposition \ref{lowerbound2}. 
\end{proof}
\noindent
As a consequence, we end up at a yet different representation of the spherical cap discrepancy, which allows us to get rid of absolute values:
\begin{cor}\label{capdefcor}
One has that
\[
\Delta \left(X\right) = \sup_{w\in\mathbb{S}^{d-1},\,t\in \left[ -1,1\right]}\mu ^{emp}\left( X,w,t\right)
-\mu ^{cap}\left(t\right)\quad \forall X\in\left(\mathbb{S}^{d-1}\right)^N.
\]
\end{cor}
\begin{proof}
Clearly, the relation '$\geq$' in the claimed equality holds true by \eqref{capdef}. On the other hand, by Proposition \ref{empgreatercap}, there exists $(w^*,t^*)\in\mathbb{S}^{d-1}\times [-1,1]$ such that 
$\Delta (X)= \mu^{emp}(X,w^*,t^*) - \mu^{cap}(t^*)$. Hence, the reverse relation '$\leq$' holds also true in the claimed equality.
\end{proof}
\noindent
Throughout the paper, we understand the sphere $\mathbb{S}^{d-1}$ as a metric space inheriting its metric from the Euclidean norm in $\mathbb{R}^{d}$. Next, we are going to prove that the spherical cap discrepancy is continuous.
\begin{theorem}\label{continuity} The function
$\Delta:\left(\mathbb{S}^{d-1}\right)^N \rightarrow \mathbb{R}$ is continuous.
\end{theorem}
\begin{proof}
We show first that $\Delta$ is lower semicontinuous. Fix arbitrary $X=\left(x^{(1)},\ldots ,x^{(N)}\right)\in\left(\mathbb{S}^{d-1}\right)^N$ and $\varepsilon >0$. According to 
Proposition~\ref{maxcap} and Proposition~\ref{empgreatercap},
there exist $w^*\in\mathbb{S}^{d-1}$ and $t^*\in [-1,1]$ such that 
\[
\Delta \left(X\right)= \mu ^{emp}\left( X,w^*,t^*\right)
-\mu ^{cap}\left(t^*\right)  .
\]

We claim that $t^*>-1$. Indeed, if $t^*=-1$, then 
$\mu ^{emp}\left( X,w^*,t^*\right)
=\mu ^{cap}\left(t^*\right)=1$, 
whence the contradiction $\Delta (X)=0$ with Proposition \ref{empgreatercap}. Define $I:=\left\{ i\in \left\{
1,\ldots ,N\right\} |x^{(i)}\in H(w^*,t^*)\right\}$. Clearly, we find $c>0$ such that
\[
t^*-c\geq -1;\quad
\langle w^*,x^{(i)}\rangle > t^*-c\quad\forall i\in I;\quad
\langle w^*,x^{(i)}\rangle < t^*-c\quad\forall i\in I^c; \quad \mid\mu^{cap}(t^*)-\mu^{cap}(t^*-c)\mid <\varepsilon.
\]
By continuity, there exists $\delta >0$ such that for all $\tilde{X}\in\left(\mathbb{S}^{d-1}\right)^N$ with $\|\tilde{X}-X\|<\delta$ ($\|\cdot \|$ denoting the Euclidean norm) it holds that
\[
\langle w^*,\tilde{x}^{(i)}\rangle > t^*-c\quad\forall i\in I;\quad
\langle w^*,\tilde{x}^{(i)}\rangle < t^*-c\quad\forall i\in I^c.
\]
Consequently, $\mu^{emp}(\tilde{X},w^*,t^*-c)=\mu ^{emp}(X,w^*,t^*)$ for all such $\tilde{X}$. 
Hence, for all $\tilde{X}\in\left(\mathbb{S}^{d-1}\right)^N$ with $\|\tilde{X}-X\|<\delta$,
\[
\Delta (\tilde{X})\geq |\mu^{emp}(\tilde{X},w^*,t^*-c)-\mu^{cap}(t^*-c)|
=|\mu^{emp}(X,w^*,t^*)-\mu^{cap}(t^*)+\mu^{cap}(t^*)-\mu^{cap}(t^*-c)| > \Delta (X)-\varepsilon .
\]
Since $\varepsilon >0$ was
arbitrary, this shows the lower semicontinuity of $\Delta$ at $X$.

As for the upper semicontinuity, assume that $\Delta$ fails to be upper semicontinuous at some $X\in\left(\mathbb{S}^{d-1}\right)^N$. Then there exist some $c>0$
as well as a sequence $X_n\in\left(\mathbb{S}^{d-1}\right)^N$ with $X_n\rightarrow X$ and $\Delta(X_n) > \Delta(X) + c$.
Let $(w_n^*,t_n^*)$ be a sequence that realizes the cap discrepancies $\Delta(X_n)$.
Due to Proposition~\ref{empgreatercap} we have that
\[
\Delta(X_n) = \mu^{emp}(X_n,w_n^*,t_n^*)
- \mu^{cap}(t_n^*) \qquad \forall n\in\mathbb{N}.
\]
Since $(w_n^*,t_n^*)\in\mathbb{S}^{d-1}\times[-1,1]$,
we may assume that $(w_n^*,t_n^*)\to (w^*,t^*)\in\mathbb{S}^{d-1}\times[-1,1]$. Altogether, $(X_n,w_n^*,t_n^*)\to (X,w^*,t^*)$. With the index set $I$ introduced in the first part of this proof, one has that $\langle w^*,x^{(i)}\rangle <t^*$ for all $i\in I^c$. By continuity, there is some $n_0$ such that $\langle w_n^*,x_n^{(i)}\rangle <t_n^*$ for all $n\geq n_0$ and $i\in I^c$. This entails that
$\mu^{emp}(X_n,w_n^*,t_n^*)\leq \mu^{emp}(X,w^*,t^*)$ 
for $n\geq n_0$.  Moreover, by continuity of $\mu^{cap}$, we have $|\mu^{cap}(t^*) - \mu^{cap}(t_n^*)| \leq c$ for sufficient large $n$.
Consequently, there exists some $n_1\in\mathbb{N}$ such that for all $n\geq n_1$
\begin{eqnarray*}
  \Delta(X_n) &=& \mu^{emp}(X_n,w_n^*,t_n^*) - \mu^{cap}(t_n^*)
              \,\leq\, \mu^{emp}(X,w^*,t^*) - \mu^{cap}(t^*) 
                 + \mu^{cap}(t^*) -\mu^{cap}(t_n^*)\\
              &\leq&  |\mu^{emp}(X,w^*,t^*) - \mu^{cap}(t^*)|
                 + |\mu^{cap}(t^*) - \mu^{cap}(t_n^*)| \leq \Delta(X) + c
\end{eqnarray*}
which is a contradiction to the previously established inequality $\Delta(X_n) > \Delta(X) + c$.
\end{proof}
\noindent
A consequence of the continuity property
is the existence of an optimal quantization with respect to the spherical cap discrepancy for any fixed number of points on the unit sphere.
\begin{cor}\label{mincap}
For each $N\geq 1$, there exists a point set
$X_*=\big(x_*^{(1)},\ldots,x_*^{(N)}\big)\in\left(\mathbb{S}^{d-1}\right)^N$ realizing the minimal spherical cap discrepancy, i.e., 
\[
\Delta(X_*) \,= \inf_{X\in\left(\mathbb{S}^{d-1}\right)^N} \Delta(X).
\]
\end{cor}

\section{Generic point sets and a representation formula for the spherical cap discrepancy}

Our ultimate goal in this paper is to prove the local Lipschitz continuity of the spherical cap discrepancy. While it is not clear at this point, whether a general Lipschitz result holds true in general, we will be able to derive it for the class of generic point sets, which would be the typical outcomes of Monte-Carlo sampling on the sphere.
\begin{definition}\label{generic}
A point set $X=\left(x^{(1)},\ldots,x^{(N)}\right)\in\left(\mathbb{R}^{d}\right)^N$
is called {\bf generic} if for any index set $I\subseteq\{1,\ldots,N\}$ with
$\# I\leq d$ the selection $\left\{x^{(i)}\,|\,i\in I\right\}$ is linear independent in $\mathbb{R}^d$.
\end{definition}
\noindent
Clearly, all point sets close enough to some generic point set
are generic themselves, which allows for the following proposition.
\begin{proposition}\label{genrem}
If $\bar{X}\in\left(\mathbb{R}^{d}\right)^N$ is generic, then there exists a neighborhood $\mathcal{O}$ of $\bar{X}$ such that $X$ is generic for each $X\in\mathcal{O}$.
\end{proposition}
\begin{definition}\label{basicdef}
Define the family of index sets
\begin{equation}\label{Phi}
\Phi:=
\big\{I\subseteq\{1,\ldots,N\}\,\big|\, 1\leq \#I \leq d\,\big\}.
\end{equation}
For generic $X=\left(x^{(1)},\ldots,x^{(N)}\right)\in\left(\mathbb{R}^{d}\right)^N$and $I\in\Phi$, let $X_I$ be the matrix whose columns are the $x^{(i)}$ ($i\in I$). Put
\begin{equation}\label{tIwI}
\mathbf{1}:=(1,\ldots,1)^\top\in\mathbb{R}^{\#I},\quad
t_I := \left(\mathbf{1}^\top \left(X_I^\top X_I\right)^{-1}\mathbf{1}\right)^{-1/2},
\quad
w_I := t_I X_I\left(X_I^\top X_I\right)^{-1}\mathbf{1},
\end{equation}
which are well-defined by the assumed genericity of $X$.
\end{definition}
\begin{proposition}\label{normed}
If $X\in\left(\mathbb{R}^{d}\right)^N$ is generic, then $t_I>0$, $w_I\in\mathbb{S}^{d-1}$ and $X_I^Tw_I=t_I{\bf 1}$
for all $I\in\Phi$. If, moreover, $X\in\left(\mathbb{S}^{d-1}\right)^N$, then $0< t_I\leq 1$ for all $I\in\Phi$.
\end{proposition}
\begin{proof}
The first assertion is evident from \eqref{tIwI}. If $X\in\left(\mathbb{S}^{d-1}\right)^N$, then the first assertion implies the second one:
\[
t_I=\langle x^{(1)},w_I\rangle\leq\|w_I\|= 1\quad\forall I\in\Phi.
\]
\end{proof}
\noindent
Next, we shall prove a representation formula for the spherical cap discrepancy of generic point sets which follows from and simplifies the enumerative formula for general point sets proven in \cite[Theorem 1]{heitsch-henrion}.
\begin{theorem}\label{genform}
Let $X=\left(x^{(1)},\ldots,x^{(N)}\right)\in\left(\mathbb{S}^{d-1}\right)^N$ be generic. Then, with the notation from Definition \ref{basicdef}, the spherical cap discrepancy may be represented as
\begin{equation}\label{enumeration}
\Delta(X) = \max_{I\in\Phi}\max\big\{\mu^{emp}(X,w_I,t_I)-\mu^{cap}(t_I),
 \mu^{emp}(X,-w_I,-t_I) - \mu^{cap}(-t_I) \big\}.
\end{equation}
\end{theorem}
\begin{proof}
For some $I\in\Phi$, denote by $\tilde{X}_I$ the extension $\tilde{X}_I=\left(X_I\atop-\mathbf{1}^T\right)$ of the matrix $X_I$.
From the enumeration formula in \cite[Theorem 1]{heitsch-henrion}
we know that the cap discrepancy is represented as a maximum of local discrepancies
associated with index subsets contained in $\Phi$.
Let $I^*\in\Phi$ some index set realizing
this maximum. Then, according to \cite[Theorem 1]{heitsch-henrion}, we have that ${\rm rank}\,\tilde{X}_{I^*}=\#I^*$,
$\gamma:=\mathbf{1}^\top\left(\tilde{X}_{I^*}^\top \tilde{X}_{I^*}\right)^{-1}\mathbf{1}\in (0,1]$ and
\begin{equation}\label{local-delta}
\Delta(X) = \max\big\{|\mu^{emp}(X,w^*,t^*)-\mu^{cap}(t^*)|,
|\mu^{emp}(X,-w^*,-t^*)-\mu^{cap}(-t^*)|\big\},
\end{equation}
where $t^*:=\left(\frac{1-\gamma}{\gamma}\right)^{1/2}\geq 0$ and
\begin{equation}\label{oldtheorem}
w^*:=\frac{1+(t^*)^2}{t^*} X_{I^*}\left(\tilde{X}_{I^*}^\top\tilde{X}_{I^*}\right)^{-1}\mathbf{1}\,\mbox{ if } t^*>0,\mbox{ and }\,
w^*\in\mathrm{Ker}X_{I^*}^\top\cap\mathbb{S}^{d-1}\,\mbox{ if }t^*=0.
\end{equation}
As noted in \cite[Theorem 1]{heitsch-henrion}, the choice of $w^*$ in the second case of \eqref{oldtheorem} is arbitrary.

%Assume now, that the maximum in \eqref{local-delta} is realized by the first term. Then, by virtue of Proposition \ref{empgreatercap},
%\begin{eqnarray}\label{removeabs}
%\Delta(X)&=& |\mu^{emp}(X,w^*,t^*)-\mu^{cap}(t^*)|=\mu^{emp}(X,w^*,t^*)-\mu^{cap}(t^*)\nonumber\\
%&=&\max\{\mu^{emp}(X,w^*,t^*)-\mu^{cap}(t^*),\mu^{emp}(X,-w^*,-t^*)-\mu^{cap}(-t^*)\}.
%\end{eqnarray}
%The same identity can be derived, if the second term in \eqref{local-delta} is dominating, hence \eqref{removeabs} holds generally true.

Then, by virtue of Proposition \ref{empgreatercap}, regardless of whether the first or the second term in \eqref{local-delta} is dominating,
\begin{eqnarray}\label{removeabs}
\Delta(X)&=&\max\{\mu^{emp}(X,w^*,t^*)-\mu^{cap}(t^*),\mu^{emp}(X,-w^*,-t^*)-\mu^{cap}(-t^*)\}.
\end{eqnarray}

To proceed, put
\[
v := -\big(1+(t^*)^2\big) \left(\tilde{X}_{I^*}^\top\tilde{X}_{I^*}\right)^{-1} \mathbf{1}.
\]
From here, we get the two relations
\[
\mathbf{1}^\top v = -\big(1+(t^*)^2\big)\gamma = - \left(1+\frac{1-\gamma}{\gamma}\right)\gamma = -1;\quad \tilde{X}_{I^*}^\top\tilde{X}_{I^*} v = - \big(1+(t^*)^2\big)\mathbf{1}.
\]
Along with the definition of $\tilde{X}_{I^*}$ as an extended matrix, this yields that
\[
- \big(1+(t^*)^2\big)\mathbf{1}=
\tilde{X}_{I^*}^\top\tilde{X}_{I^*}v= (X_{I^*}^\top X_{I^*} + \mathbf{1}\,\mathbf{1}^\top)v= X_{I^*}^\top X_{I^*}v - \mathbf{1}.
\]
Therefore, it holds $- (t^*)^2\mathbf{1} = X_{I^*}^\top X_{I^*}v$.
Since $X_{I^*}^\top X_{I^*}$ is regular by genericity of $X$, one gets that
\begin{equation}\label{vrep}
v=-(t^*)^2\left(X_{I^*}^\top X_{I^*}\right)^{-1}\mathbf{1}\quad\mbox{and}\quad 1=-{\bf 1}^Tv = (t^*)^2\mathbf{1}^T\left(X_{I^*}^\top X_{I^*}\right)^{-1}\mathbf{1}.
\end{equation}
In particular, it must be $t^* > 0$ and we observe that
\[
t^* = \left(\mathbf{1}^\top \left(X_{I^*}^\top X_{I^*}\right)^{-1}\mathbf{1}\right)^{-1/2}.
\]
Furthermore, thanks to $t^*>0$, on the other hand, by \eqref{oldtheorem} and \eqref{vrep}
one arrives at
\[
w^*=\frac{1+(t^*)^2}{t^*} X_{I^*}\left(\tilde{X}_{I^*}^\top\tilde{X}_{I^*}\right)^{-1}\mathbf{1}=-\frac{1}{t^*} X_{I^*}v=t^*X_{I^*}\left(X_{I^*}^\top X_{I^*}\right)^{-1}\mathbf{1}.
\]
Altogether, we conclude that 
$(w^*,t^*) = (w_{I^*},t_{I^*})$
for $w_{I^*},t_{I^*}$ defined in \eqref{tIwI}. Combining this with \eqref{removeabs}, we get that
\[
\Delta (X)=\max\big\{\mu^{emp}(X,w_{I^*},t_{I^*})-\mu^{cap}(t_{I^*}),
\mu^{emp}(X,-w_{I^*},-t_{I^*})-\mu^{cap}(-t_{I^*})\big\}.
\]
Moreover, because $I^*\in\Phi$, it even holds that
\[
\Delta(X)\leq\max_{I\in\Phi}\max\big\{\mu^{emp}(X,w_I,t_I)-\mu^{cap}(t_I),
\mu^{emp}(X,-w_I,-t_I)-\mu^{cap}(-t_I)\big\}\leq\Delta (X),
\]
where the last inequality relies on \eqref{capdef} and on the fact that $w_I\in\mathbb{S}^{d-1}$ and $t_I\in \left[ -1,1\right]$ for all $I\in\Phi$ by Proposition \ref{normed}. This proves \eqref{enumeration}.
\end{proof}
\noindent
We may slightly improve the representation formula \eqref{enumeration} by excluding singletons from the index family $\Phi$ in Theorem~\ref{genform}. 
\begin{proposition}
Let be $N\geq 2$. Then the assertion of Theorem~\ref{genform}
remains valid if replacing the family of index sets $\Phi$ in \eqref{Phi} by
the (smaller) family of index sets
\begin{equation}\label{barPhi}
\bar\Phi := \big\{I\subseteq\{1,\ldots,N\}\,\big|\,2\leq\# I\leq d\big\}.
\end{equation}
\end{proposition}
\begin{proof}
Since $\bar\Phi \subseteq\Phi$, it is sufficient to show that
there always exists some $\bar{I}^*\in \bar\Phi$ realizing the cap discrepancy $\Delta(X)$
in \eqref{enumeration}.
Assuming to the contrary that
\begin{equation}\label{strictdisc}
\Delta(X) > \max_{I\in\bar\Phi}\max\big\{\mu^{emp}(X,w_{I},t_{I})-\mu^{cap}(t_{I}),
\mu^{emp}(X,-w_{I},-t_{I})-\mu^{cap}(-t_{I})\big\},
\end{equation}
$\Delta(X)$ must be realized by some
$I^*\in\Phi\setminus\bar\Phi$. This implies that $I^*$
is a singleton, i.e., $I^*=\{\ell\}$ for some $\ell\in\{1,\ldots,N\}$.
Then, by \eqref{tIwI} we have $t_{I^*}=1$ and $w_{I^*}=x^{(\ell)}$, 
which by $\|x^{(i)}\|=1$  for $i=1,\ldots ,N$ implies that 
\[
x^{(i)}\in H(w_I^*,t_I^*)\Longleftrightarrow \langle x^{(i)},x^{(\ell)}\rangle =1\Longleftrightarrow x^{(i)}=x^{(\ell)}\quad (i=1,\ldots ,N).
\]
On the other hand, by genericity of $X$, we know that $x^{(i)}\neq x^{(\ell)}$ for $i\neq\ell$. Consequently, 
$\mu^{emp}(X,w_{I^*},t_{I^*})=N^{-1}$ and $\mu^{emp}(X,-w_{I^*},-t_{I^*})=1$ due to \eqref{empcomp}. Moreover, $\mu^{cap}(t_{I ^*})=\mu^{cap}(1)=0$ and $\mu^{cap}(-t_{I^*})=\mu^{cap}(-1)=1$.
Thus, 
\begin{equation}\label{contra}
\Delta(X) = \max\big\{\mu^{emp}(X,w_{I^*},t_{I^*})-\mu^{cap}(t_{I^*}),
\mu^{emp}(X,-w_{I^*},-t_{I^*})-\mu^{cap}(-t_{I^*})\big\}= N^{-1}.
\end{equation}
Consider $\bar I:=\{1,2\}\in\bar\Phi$ and define $t_{\bar I}$, $w_{\bar I}$ as in \eqref{tIwI}.
For 
\[
\bar\Delta := \max\big\{\mu^{emp}(X,w_{\bar I},t_{\bar I})-\mu^{cap}(t_{\bar I}),
\mu^{emp}(X,-w_{\bar I},-t_{\bar I})-\mu^{cap}(-t_{\bar I})\big\}
\]
it holds that (similarly to the proof of Proposition~\ref{lowerbound2})
\begin{eqnarray*}
2 \bar\Delta &\geq& \mu^{emp}(X,w_{\bar I},t_{\bar I})-\mu^{cap}(t_{\bar I}) + \mu^{emp}(X,-w_{\bar I},-t_{\bar I})-\mu^{cap}(-t_{\bar I})\\& = & 1 + N^{-1}\cdot\#\{i\in\{1,\ldots,N\}\,| \langle w_{\bar I}, x^{(i)}\rangle = t_{\bar I}\,\} -1 \geq 2 N^{-1}.
\end{eqnarray*}
From \eqref{tIwI}, it follows that $X^T_{\bar{I}}w_{\bar{I}}=t_{\bar{I}}\mathbf{1}$, and so, $\langle w_{\bar I}, x^{(i)}\rangle = t_{\bar I}$ for $i=1,2$. Therefore, $2\bar\Delta\geq 2N^{-1}$. On the other hand, $\Delta (X)>\bar\Delta$ by \eqref{strictdisc}. This yields the contradiction $\Delta (X)>N^{-1}$ with \eqref{contra}.
\end{proof}
At the end of this section, we prove a lemma connected with Theorem \ref{genform} and the quantities defined in \eqref{tIwI} which will be of later use.
\begin{lemma}\label{indexexpansion}
Let $X\in\left(\mathbb{R}^{d}\right)^N$ be generic and
$I\in\Phi$ with $\#I < d$. If there exists some index $j\in\{1,\ldots,N\}\setminus I$
such that $\langle w_I, x^{(j)}\rangle = t_I$, then for $J:=I\cup\{j\}$ it holds that $t_J = t_I$ and $w_J = w_I$.
\end{lemma}
\begin{proof}
By assumption and by definition of $w_I$ we obtain for $y:=x^{(j)}$ that
\[
    t_I = \langle w_I, y\rangle = t_I \mathbf{1}^\top (X_I^\top X_I)^{-1} X_I^{\top} y.
\]
Hence, with $t_I>0$ (see Proposition~\ref{normed}), we observe that
\begin{equation}\label{indexexpansion1}
    \mathbf{1}^\top (X_I^\top X_I)^{-1} X_I^{\top} y = y^\top X_I (X_I^\top X_I)^{-1} \mathbf{1} = 1.
\end{equation}
We first show that $t_J = t_I$: The genericity of $X$ ensures that $X_J^\top X_J$ is regular and that
\[
   \frac{1}{t_J^2} = \mathbf{1}^\top (X_J^{\top}X_J)^{-1} \mathbf{1}
        = (\mathbf{1}^\top\, |\,1)
         \left(\begin{array}{cc} X_I^\top X_I &  X_I^\top y\\
                                  y^\top X_I  &  \|y\|^2 \end{array}\right)^{-1}
                   \left(\begin{array}{c} \mathbf{1} \\ 1\end{array}\right).
\]
Using the Schur complement
$S := \|y\|^2 - y^\top X_I(X_I^\top X_I)^{-1} X_I^\top y \neq 0$ of $X_I^TX_I$, a well-known formula for the inverse of partitioned matrices, yields together with \eqref{indexexpansion1} and the definition of $w_I$ in \eqref{tIwI} that
\begin{eqnarray*}
  \frac{1}{t_J^2} &=&(\mathbf{1}^\top\,|\,1)
         \left(\begin{array}{cc}
(X_I^\top X_I)^{-1} + \frac{1}{S}(X_I^\top X_I)^{-1} X_I^\top y y^\top X_I (X_I^\top X_I)^{-1}&
         - \frac{1}{S} (X_I^\top X_I)^{-1} X_I^\top y\\ \\
                                -\frac{1}{S} y^\top X_I  (X_I^\top X_I)^{-1}&  \frac{1}{S} \end{array}\right)
                   \left(\begin{array}{c} \mathbf{1} \\ 1\end{array}\right)\\ \\
   &=& \mathbf{1}^\top(X_I^\top X_I)^{-1}\mathbf{1} 
       + \frac{1}{S}\mathbf{1}^\top (X_I^\top X_I)^{-1} X_I^\top y y^\top X_I (X_I^\top X_I)^{-1}\mathbf{1}-\frac{2}{S}\mathbf{1}^\top (X_I^\top X_I)^{-1} X_I^\top y+\frac{1}{S}=\frac{1}{t_I^2}.
\end{eqnarray*}
Thus, $t_J = t_I$. Now we show that also $w_J = w_I$:  To this end, referring to \eqref{tIwI} and taking into account \eqref{indexexpansion1}, we compute
\begin{eqnarray*}
   \langle w_J, w_I \rangle &=& t_J t_I \left(\mathbf{1}^\top\,|\,1\right)
      \left(\begin{array}{cc} X_I^\top X_I &  X_I^\top y\\
                                  y^\top X_I  &  \|y\|^2 \end{array}\right)^{-1}
      \left(\begin{array}{c} X_I^\top \\ y^\top\end{array}\right)
      X_I (X_I^\top X_I)^{-1} \mathbf{1}\\ \\
      & = & t_J t_I \left(\mathbf{1}^\top\,|\,1\right)
         \left(\begin{array}{cc} X_I^\top X_I &  X_I^\top y\\
                                  y^\top X_I  &  \|y\|^2 \end{array}\right)^{-1}
                   \left(\begin{array}{c} \mathbf{1} \\ 1\end{array}\right)\, = \, \frac{t_I t_J}{t_J^2} \,=\, 1.
\end{eqnarray*}
Since $w_I,w_J\in\mathbb{S}^{d-1}$ by Proposition \ref{normed},
we conclude that $w_J = w_I$.
\end{proof}

\section{Local Lipschitz continuity of the spherical cap discrepancy at generic point sets}

In this section, we are going to prove the main result of this paper, namely the local Lipschitz continuity of the spherical cap discrepancy $\Delta$ around generic point sets. The main argument would
aim at representing $\Delta$ as a continuous selection of $\mathcal{C}^1$-functions. The Lipschitz
continuity would allow one to calculate the Clarke subdifferential of $\Delta$ and to exploit it in the derivation of necessary optimality conditions for minimizing $\Delta$ as a function of the point set (optimal quantization). A technical difficulty arising in this context is the fact that both, the argument of deriving Lipschitz continuity for continuous selections of $\mathcal{C}^1$-functions and the definition of Clarke's subdifferential are tied to a structure of normed linear spaces, whereas $\Delta$ is defined on the sphere. For this reason, we  
introduce a generalized cap discrepancy $\Lambda$ that extends the spherical cap discrepancy $\Delta$ to arbitrary point sets
in the Euclidean space $\left(\mathbb{R}^d\right)^N$ in a 
neighborhood of a given generic point set on the unit sphere.
The idea is to prove the local Lipschitz continuity of $\Lambda$ first and then to get as an immediate corollary the same property for the genuine spherical cap discrepancy $\Delta$ which is the restriction of $\Lambda$ to the sphere around generic point sets. 

\subsection{Definition and continuity of the generalized cap discrepancy}

In order to define the generalized cap discrepancy $\Lambda$ mentioned above, one could be tempted to directly extend the definition 
\eqref{capdef} of $\Delta$ to arbitrary Euclidean
point sets. For deriving the desired Lipschitz property, however, it is beneficial to restrict considerations to generic point sets and to take the representation formula \eqref{enumeration} in Theorem~\ref{genform} as a basis for defining $\Lambda$. From now on, we shall assume that $d\geq 3$ which is no substantial restriction because uniformity of point sets on a circle is trivial.

We start by introducing an extended cap measure $\mu^{Cap}:\mathbb{R}\to\mathbb{R}$ (for dimension $d\geq 3)$ in a way that it is continuously differentiable on $\mathbb{R}$ and coincides with the original cap measure $\mu^{cap}$ from \eqref{capmeasure} on $[-1,1]$ (which is continuously differentiable on $(-1,1)$). This is achieved by the following definition:
\begin{equation}\label{capextension}
\mu^{Cap}(t):=\left\{\begin{array}{cl}
\mu^{cap}(t),&t\in [-1,1],\\
-\frac12 t + \frac12,&|t|>1,\, d=3,\\
0,&t>1,\, d\geq 4,\\
1,&t<-1,\, d\geq 4.
\end{array}\right.
\end{equation}
Indeed, it is easily seen from \eqref{capmeasure} that $\left(\mu^{cap}\right)'(-1)=\left(\mu^{cap}\right)'(1)=0$, whenever $d\geq 4$. Hence the constant continuation by the respective function values $\mu^{cap}(1)=0,\, \mu^{cap}(-1)=1$ yields a continuously differentiable extension in this case. The special case $d=3$ cannot be treated in the same way because one easily sees that 
$\mu^{cap}(t)=-t/2+1/2$ for all $t\in [-1,1]$, so that the derivatives do not vanish at -1 and 1, respectively. We may therefore simply keep the definition of the function globally in order to end up at a continuously differentiable extension. Note also, that in the case of $d=2$ (which we excluded), there exists no continuously differentiable extension of $\mu^{cap}$ because its derivative converges to $-\infty$ with the argument $t$ converging to $\pm 1$.

It is easy to show that for all $d\geq 3$ we may extend relation \eqref{capcomp} to
\begin{equation}\label{minuscap}
    \mu^{Cap}(-t) = 1 - \mu^{Cap}(t)\qquad\forall t\in\mathbb{R}.
\end{equation}
Similarly to the cap measure, we may extend the empirical measure 
%$\mu^{Emp}$
to arbitrary point sets by putting
\begin{equation}\label{empextend}
\mu ^{Emp}\left( X,w,t\right):=N^{-1}\cdot \#\left\{ i\in \left\{
1,\ldots ,N\right\} |\,x^{(i)}\in H(w,t)\right\}\quad\forall X\in\left(\mathbb{R}^{d}\right)^N\,\,\forall (w,t)\in\mathbb{R}^{d+1}.
\end{equation}
Clearly, for all normalized point sets $X\in\left(\mathbb{S}^{d-1}\right)^N$ it holds that
\begin{equation}\label{muempid}
\mu ^{Emp}\left( X,w,t\right)=\mu ^{emp}\left( X,w,t\right)\quad\forall (w,t)\in\mathbb{R}^{d+1}.
\end{equation}
For the following definition, we make reference to the quantities $t_I,w_I$ defined in \eqref{tIwI}
for $I\in\Phi$ with $\Phi$ as in \eqref{Phi}. Note that, in the previous section, all results were formulated for a fixed (generic) point set $X$. Therefore, for notational convenience, we did not emphasize the dependence of $t_I,w_I$ on $X$. In this section, however, we will investigate Lipschitz continuity of the spherical cap discrepancy based on the representation formula \eqref{enumeration}. Since now the point set will become a true variable, we will rather use the notations $t_I(X),w_I(X)$ in the definitions \eqref{tIwI} in order to stress the dependence on $X$. It is obvious that $t_I,w_I$ are continuous mappings on the set of generic point sets $X$. 
\begin{definition}\label{lambda}
For generic $X\in\left(\mathbb{R}^{d}\right)^N$, we define the generalized cap discrepancy
\begin{equation}\label{formula_lambda}
\Lambda(X)\,:=\,
\max\limits_{I\in \Phi}\max \left\{ \mu^{Emp}(X,w_{I}(X),t_{I}(X)) - \mu^{Cap}(t_{I}(X)),
\mu^{Emp}(X,-w_{I}(X),-t_{I}(X))-\mu^{Cap}(-t_{I}(X))\right\}.
\end{equation}
\end{definition}
\noindent
Thanks to Proposition \ref{genrem}, we make the following observation:
\begin{remark}\label{remgen}
If $\bar{X}\in\left(\mathbb{R}^{d}\right)^N$ is generic, then there exists a neighborhood $\mathcal{O}$ of $\bar{X}$ such that $\Lambda$ is defined on $\mathcal{O}$ and has the representation \eqref{formula_lambda} for all $X\in\mathcal{O}$.
\end{remark}
By Proposition \ref{normed}, it follows for generic $X\in\left(\mathbb{S}^{d-1}\right)^N$, that $|t_I|\leq 1$ for all $I\in\Phi$. This entails that $\mu^{Cap}(\pm t_I)=\mu^{cap}(\pm t_I)$ for all $I\in\Phi$. Moreover, by \eqref{muempid}, one also has in this case that $\mu ^{Emp}\left( X,\pm w_I,\pm t_I\right)=\mu ^{emp}\left( X,\pm w_I,\pm t_I\right)$ for all $I\in\Phi$. Now, \eqref{formula_lambda} and Theorem \ref{genform} entail that our generalized cap discrepancy reduces to the original spherical cap discrepancy for generic point sets on the sphere:
\begin{cor}\label{restriction}
For generic $X\in\left(\mathbb{S}^{d-1}\right)^N$ one has that $\Lambda(X) = \Delta(X)$.
\end{cor}
\noindent
The first basic ingredient for proving the local Lipschitz continuity of $\Lambda$ around a generic point set is the continuity itself at such point. Adding to this property later that $\Lambda$ is a selection of $\mathcal{C}^1$- functions, we will arrive at the desired Lipschitz result. Given the already proven continuity of the genuine discrepancy $\Delta$ at arbitrary point sets (Theorem \ref{continuity}), the following result shows the continuity of the generalized cap discrepancy $\Lambda$ at generic point sets. 
\begin{proposition}\label{lcontinuity} 
Let $\bar{X}\in\left(\mathbb{R}^{d}\right)^N$ be generic and $\mathcal{O}$ some open
neighborhood of $\bar{X}$ such that all $X\in\mathcal{O}$ are generic too (see Proposition \ref{genrem}). Then,
$\Lambda:\mathcal{O} \rightarrow \mathbb{R}$ is continuous.
\end{proposition}
\begin{proof}
Of course, it is sufficient to prove continuity of $\Lambda$ at the arbitrarily fixed generic point set $\bar{X}$ which entails continuity on the whole neighbourhood $\mathcal{O}$ mentioned in the statement of Proposition \ref{lcontinuity}. We shall show first the lower and later the upper semicontinuity of $\Lambda$ at $\bar{X}$, thus proving continuity itself. Let $I^*\in\Phi$ be some index set realizing the maximum in \eqref{formula_lambda}, so that $\Lambda (\bar{X})=\mu^{Emp}(\bar{X},w^*,t^*)-\mu^{Cap}(t^*)$ for some $(w^*,t^*)\in\pm\{(w_{I^*}(\bar{X}),t_{I^*}(\bar{X})\}$.
We fix an arbitrary $\varepsilon >0$. 
Now, by Lemma \ref{mainlemma} proven in the appendix, we can find some $\delta>0$, small enough
such that $\mathbb{B}_\delta (\bar{X})\subseteq\mathcal{O}$, and in such a way that  
choosing an arbitrary $X\in\mathbb{B}_\delta (\bar{X})$, we find $J,$ $w$ and $t$ with
$J\in\Phi$, $(w,t)\in\pm\{(w_J(X),t_J(X))\}$ satisfying 
%(cf. \eqref{intermed})
\[
\mu^{Emp}(X,w,t) = \mu^{Emp}(\bar{X},w^*,t^*)\quad\mbox{and}\quad
\mu^{Cap}(t)<\mu^{Cap}(t^*)+\varepsilon.
\]
In particular, by \eqref{formula_lambda}, then
%and select $\delta$ according to \eqref{intermed}. We may assume $\delta >0$ small enough such that $\mathbb{B}_\delta (\bar{X})\subseteq\mathcal{O}$ for the open neighborhood from the statement of Proposition \ref{lcontinuity}. In particular, all $X\in\mathbb{B}_\delta (\bar{X})$ are generic and so $\Lambda$ is well defined on that ball. Now, by Lemma \ref{mainlemma} proven in the appendix, we may assume $\delta$ small enough such that, choosing an arbitrary $X\in\mathbb{B}_\delta (\bar{X})$, we find $J,w,t$ as indicated in \eqref{intermed}. By \eqref{formula_lambda},
\[
\Lambda (X)\geq\mu^{Emp}(X,w,t)-\mu^{Cap}(t)
=\mu^{Emp}(\bar{X},w^*,t^*)-\mu^{Cap}(t^*)+\mu^{Cap}(t^*)-\mu^{Cap}(t)>\Lambda(\bar{X})-\varepsilon .
\]
This means that $\Lambda$ is lower semicontinuous at $\bar{X}$.
In order to show that $\Lambda$ is also upper semicontinuous at $\bar{X}$, we assume to the contrary that there exist some $c>0$ as well as a sequence
$X_n\to\bar{X}$ such that
\begin{equation}\label{Xn}
     \Lambda(X_n) > \Lambda(\bar{X}) + c\quad\forall n\in\mathbb{N}.
\end{equation}
For each $n\in\mathbb{N}$, choose $I_n^*\in \Phi$ and
$\left(w_n^*(X_n),t_n^*(X_n)\right)\in\pm\left\{(w_{I_n^*}(X_n),t_{I_n^*}(X_n))\right\}$
such that $\Lambda(X_n)$ is realized, i.e.,
\[
\Lambda(X_n) = \mu^{Emp}(X_n,w_n^*(X_n),t_n^*(X_n)) - \mu^{Cap}(t_n^*(X_n)).
\]
Since $\Phi$ is a finite set, there exists 
some $\emptyset\neq I^*\subseteq\{1,\ldots,N\}$ such that, upon passing to a subsequence, $I_n^* = I^*$ for all $n\in\mathbb{N}$. Once more, by passing to a subsequence,
we may assume that for all $n\in\mathbb{N}$ either
\[
a)\quad\left(w_n^*(X_n),t_n^*(X_n)\right)=\left(w_{I^*}(X_n),t_{I^*}(X_n)\right)  \quad\mbox{or}\quad
b)\quad\left(w_n^*(X_n),t_n^*(X_n)\right)=\left(-w_{I^*}(X_n),-t_{I^*}(X_n)\right).
\]
We consider just case $a)$ here (the second case being completely analogous). By continuity of $w_{I^*}$ and $t_{I^*}$, we have that
$w_n^*(X_n)\to w_{I^*}(\bar{X})$ and $t_n^*(X_n)\to t_{I^*}(\bar{X})$ as 
$n\rightarrow\infty$. From the definition in \eqref{empextend} it follows easily for continuity reasons that the empirical measure at some triple $(X,w,t)$ is always larger than or equal to the empirical measure of triples $(X',w',t')$ in a sufficiently small neighborhood of $(X,w,t)$.
Accordingly,
\[
\mu^{Emp}\left(X_n,w_n^*(X_n),t_n^*(X_n)\right)\leq \mu^{Emp}\left(\bar{X},w_{I^*}(\bar{X}),t_{I^*}(\bar{X})\right)
\]
for $n$ large enough.  Moreover, the continuity of the cap measure implies that
\[
\left|\mu^{Cap}\left(t_{I^*}(\bar{X})\right) - \mu^{Cap}\left(t_n^*(X_n)\right) \right| \leq c
\]
for sufficient large $n$. Consequently, there exists some $n_0\in\mathbb{N}$ with
\begin{eqnarray*}
  \Lambda(X_n) &=& \mu^{Emp}(X_n,w_n^*(X_n),t_n^*(X_n)) - \mu^{Cap}(t_n^*(X_n))\\
               &\leq& \mu^{Emp}(\bar{X},w_{I^*}(\bar{X}),t_{I^*}(\bar{X})) - \mu^{Cap}(t_{I^*}(\bar{X}))
                + \mu^{Cap}(t_{I^*}(\bar{X})) - \mu^{Cap}(t_n^*(X_n)) \leq   \Lambda(\bar{X}) + c
\end{eqnarray*}
for all $n\geq n_0$, which is a contradiction to inequality \eqref{Xn}.
Hence, $\Lambda$ is also upper semicontinuous at $\bar{X}$.
\end{proof}

\subsection{Local Lipschitz continuity of the generalized cap discrepancy at generic point sets}

Now we turn to the Lipschitz continuity of the generalized cap discrepancy locally around a generic point set $\bar{X}$. As before, we denote by $\mathcal{O}$ an open neighborhood of $\bar{X}$ of generic point sets. According to \eqref{formula_lambda}, we have the representation
\begin{equation}\label{maxphi}
    \Lambda(X) = \max_{I\in\Phi} \max \big\{ \varphi_I^{(1)}(X), \varphi_I^{(2)}(X) \big\}\qquad
    \forall X\in\mathcal{O},
\end{equation}
where
\begin{equation}\label{phi}
\varphi_I^{(1)}(X) := \mu^{Emp}(X,w_I(X),t_I(X)) - \mu^{Cap}(t_I(X));\quad\varphi_I^{(2)}(X) := \mu^{Emp}(X,-w_I(X),-t_I(X)) - \mu^{Cap}(-t_I(X)).
\end{equation}
As a preparatory step, we prove the following Lemma:
\begin{lemma}\label{phiconnect}Let $\bar{X}\in\left(\mathbb{R}^{d}\right)^N$ be generic and $\mathcal{O}$ some open
neighborhood of $\bar{X}$ such that all $X\in\mathcal{O}$ are generic too. Then, there exists a neighborhood
$\mathcal{U}\subseteq\mathcal{O}$ of $\bar X$ such that for all $I\in\Phi$ and all $X\in\mathcal{U}$ there holds:
\[
\begin{split}
    &\Lambda(\bar X) = \varphi_I^{(1)} (\bar X),\,\,\Lambda(X) = \varphi_I^{(1)}(X) \,\Rightarrow\,
     \mu^{Emp}(X,w_I(X),t_I(X)) = \mu^{Emp}(\bar X,w_I(\bar X),t_I(\bar X)),\\
    &\Lambda(\bar X) = \varphi_I^{(2)} (\bar X),\,\,\Lambda(X) = \varphi_I^{(2)}(X) \,\Rightarrow\,
     \mu^{Emp}(X,-w_I(X),-t_I(X)) = \mu^{Emp}(\bar X,-w_I(\bar X),-t_I(\bar X)).
\end{split}
\]
\end{lemma}
\begin{proof}
Without loss of generality, we prove just the first implication and assume it would not hold true. Then, there exist sequences $X_n\to\bar{X}$ and $I_n\in\Phi$ such that
\[
 \Lambda(\bar X) = \varphi_{I_n}^{(1)} (\bar X),\,\,\Lambda(X_n) = \varphi_{I_n}^{(1)}(X_n),\,\,
 \big|\mu^{Emp}(\bar X,w_{I_n}(\bar X),t_{I_n}(\bar X)) 
              - \mu^{Emp}(X_n,w_{I_n}(X_n),t_{I_n}(X_n))| \geq \frac1N.
\]
In the last inequality, we used the fact that the values of $\mu^{Emp}$ are multiples of $\frac1N$. Moreover, by continuity on $\mathcal{O}$ of $\mu^{Cap}\circ t_{I}$ for all $I\in\Phi$, we have that, for $n$ large enough,
\[
\left|\mu^{Cap}(t_{I_n}(\bar X) - \mu^{Cap}(t_{I_n}(X_n)\right|\leq \frac{1}{2N}
\]
whenever $\mathcal{U}$ is small enough.
Consequently, for $n$ large enough, we arrive at the following contradiction with the continuity of $\Lambda$ shown in Proposition \ref{lcontinuity}:
\[
    |\Lambda(\bar X) - \Lambda(X_n)| = \left| \varphi_{I_n}^{(1)} (\bar X) 
     - \varphi_{I_n}^{(1)}(X_n) \right|\geq\frac{1}{2N}.
\]
\end{proof}

\noindent
A natural idea to show the local Lipschitz continuity around generic points of the maximum function $\Lambda$ in \eqref{maxphi} would rely on checking the continuous differentiability or at least local Lipschitz continuity of the elementary functions $\varphi_I^{(1)},\varphi_I^{(2)}$. This, however, does not apply because these functions fail to be even continuous as a consequence of the discontinuity of $\mu^{Emp}$. The fact is illustrated for a numerical example in Figure \ref{fig:discont}.
\begin{figure}[htb]
    \centering
    \includegraphics[width=0.5\linewidth]{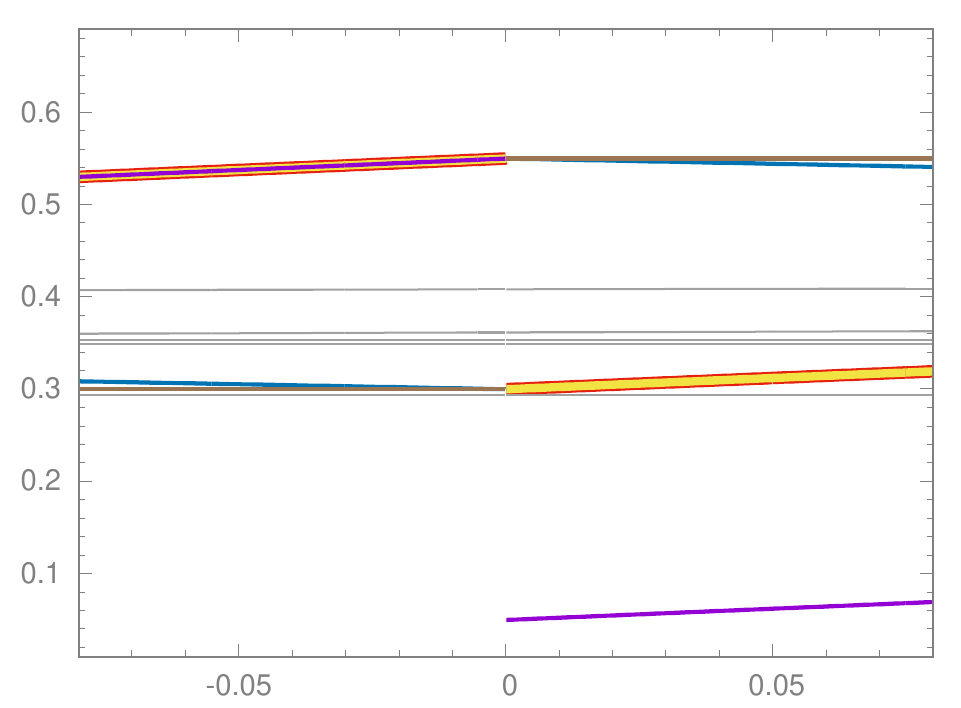}
    \caption{The cap discrepancy as a maximum of discontinuous elementary functions.}
    \label{fig:discont}
\end{figure}
Here, a generic set $\bar{X}$ of four points in $\mathbb{R}^3$ is considered and subjected to one-parametric variation (shifting one of the four points while keeping the others fixed). The variation parameter zero corresponds to the nominal point set $\bar{X}$. The figure plots the elementary functions $\varphi_I^{(1)},\varphi_I^{(2)}$ (depending on just one parameter) with those participating in the maximum of \eqref{maxphi} at $\bar{X}$ being colored.
As can be seen, the maximum $\Lambda$ of all these functions is continuous as it should be according to Proposition \ref{lcontinuity}. However, all elementary functions being active for the maximum at the nominal point set $\bar{X}$ exhibit jumps at that same point set. Still, the maximum function $\Lambda$ is apparently not only continuous but even Lipschitz continuous. To show this rigorously, we shall represent $\Lambda$ as a selection (not a maximum though!) of finitely many smooth functions. It is well known that continuous selections of smooth (or more generally: locally Lipschitzian) functions are locally Lipschitzian. The desired selection cannot be made among the original elementary functions $\varphi_I^{(1)},\varphi_I^{(2)}$ due to their discontinuity. 
We therefore define smooth modifications of these functions by locally fixing $\mu^{Emp}$ around the nominal point set $\bar{X}$:
\begin{equation}\label{tildephi}
\tilde\varphi_I^{(1)}(X) := \mu^{Emp}(\bar{X},w_I(\bar{X}),t_I(\bar{X})) - \mu^{Cap}(t_I(X));\quad
\tilde\varphi_I^{(2)}(X) := \mu^{Emp}(\bar{X},-w_I(\bar{X}),-t_I(\bar{X})) - \mu^{Cap}(-t_I(X)).
\end{equation}
Clearly, the desired smoothness of the $\tilde\varphi_I^{(1)},\tilde\varphi_I^{(2)}$ will follow 
from the continuous differentiability of the functions
$\beta_I:=\mu^{Cap}\circ t_I$ for $I\in\Phi$ around some arbitrary generic point set $X$.
\begin{lemma}\label{betderiv}
Let $\bar{X}\in\left(\mathbb{R}^{d}\right)^N$ be generic and $\mathcal{O}$ a neighbourhood of $\bar{X}$ such that all $X\in\mathcal{O}$ are generic too. Then, for each $I\in\Phi$, the function $\beta_I$ is continuously differentiable on $\mathcal{O}$ with the following partial gradients w.r.t. $x^{(l)}$, ($l\in \{1,\ldots ,N\}$):
\begin{equation}\label{partderiv}
\nabla_{x^{(l)}}\beta_I(X)=\left\{\begin{array}{ccl}0,&\mbox{if}&l\notin I
\mbox{ or }|t_I(X)|\geq 1,\, d\geq 4,\\[1ex]
-\frac12 t_I^2(X) c_I^{\tau(l)}w_I(X),&\mbox{if}&l\in I,\, d=3,\\[1ex]
\rho_I c_I^{\tau(l)}w_I(X),&\mbox{if}&l\in I,\, |t_I(X)|< 1,\, d\geq 4,
\end{array}\right.\qquad\forall X\in\mathcal{O}.
\end{equation}
Here, with $C_d$ from \eqref{cddef},
\[
\rho_I:=-C_d \, t_I^2(X) \left(1-t^2_I(X)\right)^\frac{d-3}{2}\quad\mbox{and}\quad c_I^j:=\sum\limits_{i=1}^{\#I}
(X_I^TX_I)^{-1}_{i,j}\quad (j=1,\ldots ,\#I).
\]
Moreover, for $l\in I$, the index $\tau(l)$ refers to the rank of $l$ in the index set $I$, i.e., if $I=\{\varkappa_1,\ldots ,\varkappa_{\#I}\}$, then $l=\varkappa_{\tau(l)}$.
\end{lemma}
\begin{proof}
Consider an arbitrary $X\in\mathcal{O}$ be arbitrary, whence $X$ is generic.  Let $I\in\Phi$ be arbitrary. We assume that $I=\{\varkappa_1,\ldots ,\varkappa_{\#I}\}\subseteq\{1,\ldots ,N\}$. We want to derive first the function
\begin{equation}\label{aldef}
\alpha (X):=\mathbf{1}^T\left(X_I^TX_I\right)^{-1}\mathbf{1}
\end{equation}
From well-known rules of matrix differential calculus (see, e.g., \cite{magnus99}) one obtains with $M(X):=X_I^TX_I$ that
\[
\frac{\partial\alpha}{\partial X_{k,l}}(X)=
-2\sum\limits_{q=1}^{\#I}\left(\sum\limits_{i=1}^{\#I}[M(X)]^{-1}_{i,\tau(l)}\right)\left(\sum\limits_{j=1}^{\#I}[M(X)]^{-1}_{j,q}\right)X_{k,\varkappa_q}
\]
with $\tau(l)$ as introduced in the statement of this lemma
(for a detailed argumentation, we refer to the preprint version of this paper \cite{HeiHenPre25}). By definition of $M(X)$ and of the coefficients $c_I^j$ introduced in the statement of this Lemma, we obtain that
\begin{equation}\label{alderiv}
\nabla_{x^{(l)}}\alpha (X)=-2c_I^{\tau(l)}\sum\limits_{q=1}^{\#I}c_I^qx^{(\varkappa_q)}
\end{equation}
Next, we observe that, for all $q=1,\ldots ,\#I$,
\[
\left[\left(X_I^TX_I\right)^{-1}\mathbf{1}\right]_q=\sum_{i=1}^{\#I}\left(X_I^TX_I\right)^{-1}_{q,i}=\sum_{i=1}^{\#I}\left(X_I^TX_I\right)^{-1}_{i,q}=c_I^q.
\]
Consequently, by definition \eqref{tIwI},
\[
w_I(X)=t_I(X) X_I\left(X_I^\top X_I\right)^{-1}\mathbf{1}=
t_I(X)\sum_{q=1}^{\#I}c_I^qx^{(\varkappa_j)}.
\]
Thanks to \eqref{alderiv}, this entails that
\begin{equation}\label{alderiv2}
\nabla_{x^{(l)}}\alpha (X)=-2\frac{c_I^{\tau(l)}}{t_I(X)}w_I(X)
\end{equation}
We observe next that the function $\mu^{cap}$ defined in \eqref{capmeasure} is continuously differentiable for $d\geq 3, t\in (0,1)$ with 
\[
[\mu^{cap}]'(t)=-C_d(1-t^2)^\frac{d-3}{2}.
\]
Along with \eqref{capextension} and the explanations below this equation, this yields that $\mu^{Cap}$ is continuously differentiable with
\begin{equation}\label{casedist}
[\mu^{Cap}]'(t)=\left\{\begin{array}{ccl}0,&\mbox{if}&|t_I(X)|\geq 1, d\geq 4,\\[1ex]
-\frac{1}{2},&\mbox{if}&d=3,\\[1ex]
-C_d(1-t^2)^\frac{d-3}{2},&\mbox{if}&|t_I(X)|< 1.
\end{array}\right.
\end{equation}
Moreover, the function $t_I=\alpha^{-1/2}$ defined in \eqref{tIwI} and \eqref{aldef} is continuously differentiable in the generic point set $X$ because $\alpha$ was shown so in \eqref{alderiv}.
Therefore, the function $\beta=\mu^{Cap}\circ t_I$ is continuously differentiable in $X$ with
\begin{equation}
\frac{\partial\beta_I}{\partial X_{k,l}}(X)=[\mu^{Cap}]'(t_I(X))\cdot \frac{\partial t_I}{\partial X_{k,l}}(X)=-\frac{1}{2}[\mu^{Cap}]'(t_I(X))\cdot[\alpha(X)]^{-3/2}\frac{\partial\alpha}{\partial X_{k,l}}(X),
\end{equation}
whence, along with \eqref{alderiv2},
\begin{equation}\label{betaderiv}
\nabla_{x^{(l)}}\beta_I(X)= -\frac{1}{2}[\mu^{Cap}]'(t_I(X))\cdot [t_I(X)]^3\nabla_{x^{(l)}}\alpha (X)=[\mu^{Cap}]'(t_I(X))\cdot [t_I(X)]^2
c_I^{\tau(l)}w_I(X).
\end{equation}
Now, lines two and three in \eqref{casedist}, yield the corresponding lines in \eqref{partderiv}. Clearly, the outcomes of \eqref{partderiv} depend continuously on $X$ thanks to the continuity of $t_I,w_I$. This also proves the continuous differentiability of $\beta_I$ on $\mathcal{O}$.
\end{proof}
\begin{cor}\label{phismoothnesss}
For each $I\in\Phi$, the functions $\tilde\varphi_I^{(1)}(\cdot)$, $\tilde\varphi_I^{(2)}(\cdot)$
defined in \eqref{tildephi} are continuously differentiable on $\mathcal{O}(\bar X)$ with
\[
\nabla\tilde\varphi_I^{(1)}(X)=-\nabla\beta_I(X)\quad\mbox{and}\quad\nabla\tilde\varphi_I^{(2)}(X)=\nabla\beta_I(X)\qquad\forall X\in\mathcal{O}(\bar X).
\]
\end{cor}
\begin{proof}
The first formula above is evident from the definition of $\tilde\varphi_I^{(1)}$ in \eqref{tildephi}. Similarly, the definition of $\tilde\varphi_I^{(2)}$ yields that
\[
\nabla\tilde\varphi_I^{(2)}(X)=[\mu^{Cap}]'(-t_I(X))\nabla t_I(X)=[\mu^{Cap}]'(t_I(X))\nabla t_I(X)=\nabla\beta_I(X),
\]
where the second equation follows from \eqref{casedist}.
\end{proof}

\noindent
We shall prove now that, locally around generic point sets, $\Lambda$
is a selection of the continuously differentiable functions $\tilde\varphi_I^{(1)},\tilde\varphi_I^{(2)}$.
\begin{proposition}\label{selection}
Let $\bar{X}\in\left(\mathbb{R}^{d}\right)^N$ be generic and $\mathcal{O}$ some open
neighborhood of $\bar{X}$ such that all $X\in\mathcal{O}$ are generic too. Then, there exists a neighborhood
$\mathcal{V}\subseteq\mathcal{O}$ of $\bar X$ such that
for all $X\in\mathcal{V}$ there are $I\in\Phi$ and $s\in\{1,2\}$ with
$\Lambda(X) = \tilde\varphi_I^{(s)}(X)$.
\end{proposition}
\begin{proof}
Let $\mathcal{U}\subseteq\mathcal{O}$ be the neighborhood of $\bar{X}$ from Lemma \ref{phiconnect} and define the set of active indices as
\begin{equation}\label{actind}
\mathcal{A}(X):=\big\{(I,s)\in\Phi\times\{1,2\}\mid
\Lambda(X) = \varphi_I^{(s)}(X)\big\}\quad\forall X\in\mathcal{O}.
\end{equation}
(see \eqref{maxphi}). We claim that there exists a neighborhood $\mathcal{V}\subseteq\mathcal{U}$ of $\bar{X}$
such that 
\begin{equation}\label{active}
\mathcal{A}(X)\subseteq \mathcal{A}(\bar X)\quad\forall X\in\mathcal{V}.
\end{equation}
If this wasn't the case, we could find sequences
$X_n\in\left(\mathbb{R}^d\right)^N$ and
$(I_n,s_n)\in\Phi\times\{1,2\}$
such that
\[
\Lambda(X_n) = \varphi_{I_n}^{(s_n)}(X_n),\quad
\Lambda(\bar X) > \varphi_{I_n}^{(s_n)}(\bar X) \quad\forall n\in\mathbb{N}\quad\mbox{and}\quad
X_n\rightarrow \bar{X}.
\]
Moreover, by passing to a subsequence, we may find some $\bar I \in \Phi$ and $\bar s \in\{1,2\}$ such that
\[
\Lambda(X_n) = \varphi_{\bar I}^{(\bar s)}(X_n)\quad\forall n\in\mathbb{N}\quad\mbox{and}\quad
\Lambda (\bar X) > \varphi_{\bar I}^{(\bar s)}(\bar X).
\]
Because $\Lambda$ is continuous at $\bar{X}$ by Proposition \ref{lcontinuity}, there is some $n_0\in\mathbb{N}$ such that
\[
\varphi_{\bar I}^{(\bar s)}(X_n) > \varphi_{\bar I}^{(\bar s)}(\bar X) + \frac{c}{2}\quad\forall n\geq n_0,\quad\mbox{where } c:= \Lambda (\bar X)-\varphi_{\bar I}^{(\bar s)}(\bar X)>0.
\]
Next, we use an argument already employed in the proof of Proposition \ref{lcontinuity}, namely that the definition in \eqref{empextend} easily implies for continuity reasons that the empirical measure at some triple $(X,w,t)$ is always larger than or equal to the empirical measure of triples $(X',w',t')$ in a sufficiently small neighborhood of $(X,w,t)$. Assuming, without loss of generality that $\bar{s}=1$ (the argument being exactly the same for $\bar{s}=2$), we therefore get for $n\geq n_0$ that
\begin{eqnarray*}
\mu^{Emp}(\bar X,w_{\bar I}(\bar X),t_{\bar I}(\bar X)) - \mu^{Cap}(t_{\bar I}(\bar X))+\frac{c}{2}&=&   
\varphi_{\bar I}^{(1)}(\bar X) + \frac{c}{2}<\varphi_{\bar I}^{(1)}(X_n)\\&=&
\mu^{Emp}(X_n,w_{\bar I}(X_n),t_{\bar I}(X_n)) - \mu^{Cap}(t_{\bar I}(X_n))\\
&\leq&\mu^{Emp}(\bar X,w_{\bar I}(\bar X),t_{\bar I}(\bar X)) - \mu^{Cap}(t_{\bar I}(X_n)).
\end{eqnarray*}
Passing to the limit on the right-hand side and exploiting the continuity of $\mu^{Cap}\circ t_{\bar I}$, we arrive at the contradiction
\[
-\mu^{Cap}(t_{\bar I}(\bar X))+\frac{c}{2}\leq -\mu^{Cap}(t_{\bar I}(\bar X))
\]
which proves \eqref{active}.

Now, fix an arbitrary $X\in\mathcal{V}$ and choose $I\in \Phi$ and
$s\in\{1,2\}$ such that $\Lambda(X) = \varphi_I^{(s)}(X)$. Then, $(I,s)\in\mathcal{A}(X)\subseteq\mathcal{A}(\bar X)$ by \eqref{active}, whence $\Lambda(\bar X) = \varphi_I^{(s)}(\bar X)$. 
Since $\mathcal{V}\subseteq\mathcal{U}$, Lemma \ref{phiconnect} 
yields that $\varphi_I^{(s)}(X) = \tilde\varphi_I^{(s)}(X)$. Thus, $\Lambda(X) = \tilde\varphi_I^{(s)}(X)$, as was to be shown.
\end{proof}

\noindent
We are now in a position to formulate the main result of this paper:
\begin{theorem}\label{lipschitz}
Let $\bar{X}=\left(\bar{x}^{(1)},\ldots,\bar{x}^{(N)}\right)\in\left(\mathbb{R}^{d}\right)^N$ be generic. Then, there exists 
some neighborhood $\mathcal{U}$ of $\bar{X}$ such that $X$ is generic for all $X\in\mathcal{U}$ and $\Lambda$ is Lipschitz continuous on $\mathcal{U}$. In other words, there exists some $L>0$ such that
\[
\left|\Lambda (X_1)-\Lambda (X_2)\right|\leq L\|X_1-X_2|\quad\forall
X_1,X_2\in\mathcal{U}.
\]
\end{theorem}
\begin{proof}
By Propositions \ref{lcontinuity} and \ref{selection}, there exists a neighborhood $\mathcal{U}$ of $\bar{X}$ such that $\Lambda$ is continuous and a selection of finitely many continuously differentiable functions on $\mathcal{U}$ (which means that $\Lambda$ is piecewise differentiable in the terminology of Scholtes \cite[page 91]{Scholtes2012}). In particular, $\Lambda$ is a continuous selection of Lipschitz functions on $\mathcal{U}$, hence $\Lambda$ is Lipschitz continuous on $\mathcal{U}$ itself \cite[Proposition 4.1.2.]{Scholtes2012}.
\end{proof}

\noindent As an immediate consequence, we get the desired local Lipschitz continuity of the (original) spherical cap discrepancy $\Delta$  around generic points on the sphere:
\begin{cor}\label{lipschitzcor}
Let $\bar{X}=\left(\bar{x}^{(1)},\ldots,\bar{x}^{(N)}\right)\in\left(\mathbb{S}^{d-1}\right)^N$ be generic. Then, there exists 
some neighborhood $\mathcal{U}$ of $\bar{X}$ such that $X$ is generic for all $X\in\mathcal{U}\cap\mathbb{S}^{d-1}$ and $\Delta$ is Lipschitz continuous on $\mathcal{U}\cap\mathbb{S}^{d-1}$. In other words, there exists some $L>0$ such that
\[
\left|\Delta (X_1)-\Delta (X_2)\right|\leq L\|X_1-X_2\|\quad\forall
X_1,X_2\in\mathcal{U}\cap\mathbb{S}^{d-1}.
\]
\end{cor}
\begin{proof}
This follows immediately from Theorem \ref{lipschitz} and Corollary \ref{restriction}.
\end{proof}
It is noteworthy that the Lipschitz constant $L$ in Theorem \ref{lipschitz} (which is the same as in Corollary~\ref{lipschitzcor}) can be explicitly estimated from the data by using the formulae in Lemma~\ref{betderiv}. Indeed, as a consequence of Proposition~\ref{selection} and of \cite[Proposition 4.1.2]{Scholtes2012}, we obtain that the Lipschitz constant $L$ of $\Lambda$ on $\mathcal{U}$ can be represented by the Lipschitz constants $L_i^{(s)}$ of the continuously differentiable functions $\tilde\varphi_I^{(s)}$ as
\[
L=\max\limits_{(I,s)\in\Phi\times\{1,2\}}L_i^{(s)}.
\]
Clearly, the $L_i^{(s)}$ can be chosen greater than but arbitrarily close to the norms $\|\nabla\tilde\varphi_I^{(s)}(\bar{X})\|$ by shrinking the neighbourhood $\mathcal{U}$. By Corollary \ref{phismoothnesss} and Lemma \ref{betderiv}, a rough upper estimate of the $L_i^{(s)}$ would be
\[
C_d\max\limits_{I\in\Phi, l\in I}|c_I^{\tau(l)}|
\]
(a finer estimate would incorporate the expressions $t_I^2(X)$).

\section{Optimal quantization and necessary optimality conditions}\label{secoptquant}

Finding an optimal point set on the sphere minimizing the spherical cap discrepancy amounts to the optimization problem
\begin{equation}\label{optquant}
\mbox{minimize }\Delta (X),   
\end{equation}
where $\Delta:\left(\mathbb{S}^{d-1}\right)^N\to [0,1]$ is the spherical cap discrepancy introduced in \eqref{capdef0}. 
This problem is also referred to as {\it optimal quantization} and has to be distinguished from the construction of low discrepancy sequences because the cardinality $N$ of the point set $X$ is fixed. 
Our aim is to establish necessary optimality conditions a point set has to satisfy in order to be optimal. Note that there is no hope for optimality conditions which are sufficient at the same time due to the lack of convexity of $\Lambda$. We will restrict ourselves here to generic point sets. The degenerate case seems to be more delicate to handle and is left for future research. 

While \eqref{optquant} is a free (without constraints) optimization problem, the domain of the objective function is a manifold. Standard optimization problems, however, are usually defined on normed spaces subjected to possible further constraints in order to conveniently derive nonsmooth optimality conditions by using tools from generalized differentiation such as the subdifferentials in the sense of Clarke \cite{clarke1983} or Mordukhovich \cite{Mordukhovich2006}. For this reason it is beneficial to equivalently rewrite problem \eqref{optquant} as an optimization problem in the Euclidean space with the additional constraint that the arguments belong to the sphere componentwise:
\begin{equation}\label{quant}
\mbox{minimize }\Lambda (X)\mbox{ subject to } X\in\left(\mathbb{S}^{d-1}\right)^N, X\mbox{ generic}.
\end{equation}
The restriction to generic $X$ is necessary because $\Lambda$ is defined for such point sets only. While the genericity constraint cannot be conveniently described as a classical (in-)equality constraint, it is an open property. This means, that if we are interested in checking whether some generic point set $\bar{X}$ satisfies certain necessary optimality conditions, then we don't have to care about this constraint, because we know it persists to hold in an open neighbourhood $\mathcal{O}$ of $\bar{X}$ and, thus, has no impact on the necessary optimality condition at all. Now, the equivalence of\eqref{quant} with \eqref{optquant} around some generic $X\in\left(\mathbb{S}^{d-1}\right)^N$ is evident from
Corollary \ref{restriction}.
We represent the normalization constraint on $X=\left(x^{(1)},\ldots,x^{(N)}\right)$ as the set of smooth equalities
\[
\|x^{(l)}\|^2=1\quad\forall l\in\{1,\ldots ,N\}.
\]
Then the derivative with respect to $X$ of the $l$-th constraint function equals the matrix
\[
2\big(0 \cdots 0\,|\,x^{(l)}\,|\,0\cdots 0\big)
\]
Clearly, all these derivatives are linearly independent due to $x^{(l)}\neq 0$. Now, the local Lipschitz continuity of $\Lambda$ and the continuous differentiability of the constraint functions implies that a generic point set $\bar{X}$ being a (local) solution of the optimal quantization problem \eqref{quant} has to satisfy the inclusion
\begin{equation}\label{clarkeopt}
\big(\lambda_1\bar{x}^{(1)}|\cdots |\lambda_N\bar{x}^{(N)}\big)\in\partial^C\Lambda (\bar{X})    
\end{equation}
for certain multipliers $\lambda_1,\ldots ,\lambda_N\in\mathbb{R}$, where $\partial^C\Lambda$ refers to the Clarke subdifferential of $\Lambda$ \cite[p.235-236]{clarke1983}. In order to work with such an abstract condition, one has to make the Clarke subdifferential more explicit: From \cite[Proposition 4.3.1.]{Scholtes2012}, we know that at generic $\bar{X}$ the identity
\[
\partial^C\Lambda (\bar{X})={\rm conv}\,\left\{\nabla\tilde\varphi_I^{(s)}(\bar{X})\mid (I,s)\in\mathcal{A}^*(\bar{X})\subseteq\Phi\times\{1,2\}\right\}
\]
holds true, where $\mathcal{A}^*(\bar{X})$ refers to the so-called {\it set of essentially active indices} (see \cite[p. 92]{Scholtes2012}) and 'conv' refers to the convex hull. Rather than providing a precise definition of the difficult to handle index set $\mathcal{A}^*(\bar{X})$ here, we just recall from its definition in \cite[p. 92]{Scholtes2012}, that it is always a subset of the set of active indices $\mathcal{A}(\bar{X})$ defined in \eqref{actind}
\[
\mathcal{A}^*(\bar{X})\subseteq\big\{(I,s)\in\Phi\times\{1,2\}\mid\Lambda (\bar{X}) = \tilde{\varphi}_I^{(s)}(\bar{X})\big\}=\big\{(I,s)\in\Phi\times\{1,2\}\mid\Lambda(\bar{X}) = \varphi_I^{(s)}(\bar{X})\big\}=\mathcal{A}(\bar{X}).
\]
Consequently, we arrive at an explicit upper estimate of $\partial^C\Lambda (\bar{X})$ just in terms of active gradients:
\[
\partial^C\Lambda (\bar{X})\subseteq {\rm conv}\,\left\{\nabla\tilde\varphi_I^{(s)}(\bar{X})\mid (I,s)\in\mathcal{A}(\bar{X})\right\}.
\]
This upper estimate can now be clearly used to establish a weakened but explicit necessary optimality condition as follows:
A generic point set $\bar{X}$ being a (local) solution of the optimal quantization problem \eqref{quant} has to satisfy the inclusion
\[
\big(\lambda_1\bar{x}^{(1)}|\cdots |\lambda_N\bar{x}^{(N)}\big)\in {\rm conv}\,\left\{\nabla\tilde\varphi_I^{(s)}(\bar{X})\mid (I,s)\in\mathcal{A}(\bar{X})\right\}    
\]
for certain multipliers $\lambda_1,\ldots ,\lambda_N\in\mathbb{R}$. Resolving for the convex hull, we may extend this statement to:
If a generic point set $\bar{X}$ is a (local) solution of the optimal quantization problem \eqref{quant}, then there exist multipliers $\lambda_1,\ldots ,\lambda_N\in\mathbb{R}$ and $\gamma_{(I,s)}\geq 0$ for $(I,s)\in\mathcal{A}(\bar{X})$ with
\[
\lambda_l\bar{x}^{(l)}=\sum\limits_{(I,s)\in\mathcal{A}(\bar{X})} \gamma_{(I,s)}\nabla_{x^{(l)}}\tilde\varphi_I^{(s)}(\bar{X})\quad\mbox{and}\quad\sum\limits_{(I,s)\in\mathcal{A}(\bar{X})}\gamma_{(I,s)}=1\quad (l=1,\ldots ,N). 
\]
Taking into account that $\nabla_{x^{(l)}}\tilde\varphi_I^{(s)}(\bar{X})=(-1)^s\nabla_{x^{(l)}}\beta_I^{(s)}(\bar{X})$ for $(I,s)\in\Phi\times\{1,2\}$ and $l=1,\ldots ,N$ by Corollary~\ref{phismoothnesss}, and that $\nabla_{x^{(l)}}\beta_I^{(s)}(\bar{X})=0$ if $l\notin I$ by Lemma \ref{betderiv}, we may further rewrite this last relation as
\[
\lambda_l\bar{x}^{(l)}=\sum\limits_{\{(I,s)\in\mathcal{A}(\bar{X}) \mid l\in I\}} \gamma_{(I,s)}(-1)^s\nabla_{x^{(l)}}\beta_I^{(s)}(\bar{X})\quad\mbox{and}\quad\sum\limits_{(I,s)\in\mathcal{A}(\bar{X})}\gamma_{(I,s)}=1\quad (l=1,\ldots ,N). 
\]
This relation is now fully explicit thanks to the explicit gradient formulae in Lemma \ref{betderiv} and it can be used to figure out potential candidates for local minima of the spherical cap discrepancy (by verifying the necessary optimality conditions) or to exclude certain generic point sets as local or global minima (by showing that the necessary optimality conditions cannot hold). We shall not pursue this concrete application of optimality conditions here and rather leave this to future research.

\section*{Conclusions}

We have proven the Lipschitz continuity of the spherical cap discrepancy around generic point sets on the sphere. Of course, it would be desirable to prove or disprove the Lipschitz continuity on the whole sphere. It seems that we will not be able to show the positive result using the approach taken here (via the representation formula \eqref{formula_lambda}). On the other hand, a counter example isn't easy to construct either. We therefore strongly believe that the following conjecture holds true (note that local Lipschitz continuity around arbitrary point sets implies the global Lipschitz continuity by compactness of the sphere):
\begin{conjecture}
The spherical cap discrepancy $\Delta:\mathbb{S}^{(d-1)}\to [0,1]$ is Lipschitz continuous. 
\end{conjecture}
\noindent
Apart from proving this conjecture, future research will be devoted to the concrete application of the necessary optimality conditions derived in Section \ref{secoptquant} and to a numerical solution of the optimal quantization provlem exploiting Lipschitz continuity of the spherical cap discrepancy.

\section*{Acknowledgments}

This work is supported by the German Research Foundation (DFG)
within the project B04 of CRC/Transregio 154
and by the FMJH Program Gaspard Monge in optimization and operations
research including support to this program by EDF.

%%%%%%%%%%%%

\appendix
\section{Appendix}
\begin{proposition}\label{disccoinc}
For the discrepancies presented in the introduction it holds that
$\Delta (X)=\Delta^0(X) =\Delta^1(X)$ for all $X=\left(x^{(1)},\ldots ,x^{(N)}\right)\in\left(\mathbb{S}^{d-1}\right)^N$.
\end{proposition}
\begin{proof}
In order to show that $\Delta (X)=\Delta^0(X)$, it is evidently sufficient to verify the relation 
\begin{equation}\label{dominate}
\left\vert \mu ^{emp}\left( X,w,t\right)
-\mu ^{cap}\left(w,t\right) \right\vert\leq\Delta (X)\quad\forall 
(w,t)\in\mathbb{R}^{d}\times\mathbb{R}.
\end{equation}
Let $(w,t)\in\mathbb{R}^{d}\times\mathbb{R}$ be arbitrary and assume first that $w=0$ and $t\leq 0$. Then, $H(w,t)=\mathbb{R}^{d}$ and, hence, $\mu ^{emp}\left( X,w,t\right)=\mu ^{cap}\left(w,t\right)=1$ and \eqref{dominate} follows trivially. Similarly, if $w=0$ and $t>0$, then $H(w,t)=\emptyset$ and, hence, $\mu ^{emp}\left( X,w,t\right)=\mu ^{cap}\left(w,t\right)=0$, so that \eqref{dominate} follows again. Next, let $w\neq 0$ and $t<-\|w\|$. Then, $x\in H(w,t)$ for all $x\in\mathbb{S}^{d-1}$ 
because of
\[
\langle w,x\rangle\geq -\|w\|>t\quad\forall x\in\mathbb{S}^{d-1}
\]
which implies $\mu ^{emp}\left( X,w,t\right)=\mu ^{cap}\left(w,t\right)=1$.
Similarly, if $t>\|w\|$, then $x\notin H(w,t)$ for all $x\in\mathbb{S}^{d-1}$ because any such $x$ satisfies the relation
\[
\langle w,x\rangle\leq \|w\| < t,
\]
so we have $\mu ^{emp}\left( X,w,t\right)=\mu ^{cap}\left(w,t\right)=0$.
In both cases, \eqref{dominate} follows trivially as before. It remains to consider the case that $w\neq 0$ and $|t|\leq\|w\|$. Then, $H(w,t)=H(w^*,t^*)$ for $w^*:=w/\|w\|\in\mathbb{S}^{d-1}$ and $t^*:=t/\|w\|\in [-1,1]$. Accordingly, and by virtue of \eqref{capdef0},
\[
\left\vert \mu ^{emp}\left( X,w,t\right)-\mu ^{cap}\left(w,t\right) \right\vert =\left\vert \mu ^{emp}\left( X,w^*,t^*\right)-\mu ^{cap}\left(w^*,t^*\right) \right\vert\leq\Delta(X).
\]
It remains to show that 
\[
\Delta (X)=\Delta^1(X):=\sup_{w\in\mathbb{S}^{d-1},\,t\in \left[ -1,1\right]}\left\vert \mu ^{emp}_0\left( X,w,t\right)
-\mu ^{cap}_0\left(w,t\right) \right\vert,
\]
where, with $H_0(w,t):=\left\{ x\in \mathbb{R}^{d}|\left\langle w,x\right\rangle >t\right\}$, one defines $\mu ^{cap}_0\left(w,t\right):=\sigma \left( \mathbb{S}^{d-1}\cap
H_0(w,t)\right)$ and
\[
\mu ^{emp}_0\left( X,w,t\right):=N^{-1}\cdot \#\left\{ i\in \left\{
1,\ldots ,N\right\} |x^{(i)}\in H_0(w,t)\right\}
\]
for all $w\in\mathbb{S}^{d-1}$ and $t\in \left[ -1,1\right]$. We immediately check from the definitions, that for arbitrary $(w,t)\in\mathbb{S}^{d-1}\times [-1,1]$ one has that
\[
\mu ^{emp}_0\left( X,w,t\right)=1-\mu ^{emp}\left( X,-w,-t\right).
\]
Moreover, by \eqref{capcomp}, for arbitrary $(w,t)\in\mathbb{S}^{d-1}\times [-1,1]$ one has that
\[
\mu ^{cap}_0\left(w,t\right)=\mu ^{cap}\left(w,t\right)=1-\mu ^{cap}\left(-w,-t\right).
\]
Now, the claimed equality $\Delta (X)=\Delta^1(X)$ follows readily from the identity
\[
\left\vert \mu ^{emp}_0\left( X,w,t\right)
-\mu ^{cap}_0\left(w,t\right) \right\vert =\left\vert \mu ^{emp}\left( X,-w,-t\right)
-\mu ^{cap}\left(-w,-t\right) \right\vert\quad\forall (w,t)\in\mathbb{S}^{d-1}\times [-1,1].
\]
\end{proof}
\begin{lemma}\label{mainlemma}
Let $\bar{X}\in\left(\mathbb{R}^{d}\right)^N$ be generic. Let $I^*\in\Phi$ be some index set realizing the maximum in \eqref{formula_lambda}, so that $\Lambda (\bar{X})=\mu^{Emp}(\bar{X},w^*,t^*)-\mu^{Cap}(t^*)$ for some $(w^*,t^*)\in\pm\left\{(w_{I^*}(\bar{X}),t_{I^*}(\bar{X})\right\}$. Then, the following holds true:
\begin{equation}\label{intermed}
\begin{split}
\forall \varepsilon & > 0\,\,\exists\delta > 0\,\,\forall X\in \mathbb{B}_\delta (\bar{X})\,\,\exists J\in\Phi \,\,\exists (w,t)\in\pm\{(w_{J}(X),t_{J}(X))\}:\\
&\mu^{Emp}(X,w,t) = \mu^{Emp}(\bar{X},w^*,t^*),\,\,
\mu^{Cap}(t)<\mu^{Cap}(t^*)+\varepsilon .
\end{split}
\end{equation}
\end{lemma}
\begin{proof}
We assume from the very beginning that the $\delta$ to be found in \eqref{intermed} is small enough to satisfy the inclusion $\mathbb{B}_\delta (\bar{X})\subseteq\mathcal{O}$ from Remark \ref{remgen}, so that all $X$ from this ball are generic.
We introduce the index sets
\[
I_0:=\left\{ i\in\{1,\ldots ,N\right\}\mid\,\langle w^*,\bar{x}^{(i)}\rangle =t^*\},\quad I_1:=\left\{ i\in\{1,\ldots ,N\right\}\mid\,\langle w^*,\bar{x}^{(i)}\rangle >t^*\}
\]
Proposition \ref{normed} ensures that $\bar{X}_{I^*}^Tw_{I^*}(\bar{X})=t_{I^*}(\bar{X}){\bf 1}$, whence $\bar{X}_{I^*}^Tw^*=t^*{\bf 1}$ due to $(w^*,t^*)\in\pm\{(w_{I^*}(\bar{X}),t_{I^*}(\bar{X})\}$. It follows that $I^*\subseteq I_0$. 

\noindent\underline{Case 1: $I^*=I_0$}.
Without loss of generality, we may also assume that $(w^*,t^*)=(w_{I^*}(\bar{X}),t_{I^*}(\bar{X}))$ (the opposite case following by absolutely analogous arguments). Then, by definition,
\[
\mu^{\rm Emp}(\bar{X},w_{I^*}(\bar{X}),t_{I^*}(\bar{X}))=N^{-1}(\#I_0+\#I_1).
\]
For arbitrarily given $\varepsilon >0$ we choose $\delta >0$ such that $\mathbb{B}_\delta (\bar{X})\subseteq\mathcal{O}$ for the open neighborhood from the statement of Proposition \ref{lcontinuity} (i.e. all $X\in\mathbb{B}_\delta (\bar{X})$ are generic). Moreover, $\delta >0$ is chosen small enough to satisfy (by continuity of the mappings $t_{I^*},w_{I^*}$)
\[
\langle w_{I^*}(X),x^{(i)}\rangle >t_{I^*}(X)\quad\forall
i\in I_1\quad\langle w_{I^*}(X),x^{(i)}\rangle <t_{I^*}(X)\quad\forall
i\in (I_0\cup I_1)^c
\]
for all $X\in\mathbb{B}_\delta (\bar{X})$ and all $i\in\{1,\ldots ,N\}$. Moreover, by Proposition \ref{normed}, $X_{I^*}^Tw_{I^*}(X)=t_{I^*}(X){\bf 1}$ for all such $X$, hence $\langle w_{I^*}(X),x^{(i)}\rangle =t_{I^*}(X)$ for all $i\in I^*=I_0$. Altogether, this implies that
$\langle w_{I^*}(X),x^{(i)}\rangle\geq t_{I^*}(X)$ if and only if $i\in I_0\cup I_1$. Therefore,
\[
\mu^{\rm Emp}(X,w_{I^*}(X),t_{I^*}(X))=N^{-1}(\#I_0+\#I_1)=\mu^{\rm Emp}(\bar{X},w_{I^*}(\bar{X}),t_{I^*}(\bar{X}))\quad\forall X\in\mathbb{B}_\delta (\bar{X}).
\]
Finally, by continuity of $\mu^{Cap}$, we may further shrink $\delta >0$ such that $\mu^{Cap}(t_{I^*}(X))<\mu^{Cap}(t_{I^*}(\bar{X})))+\varepsilon$ for all $X\in\mathbb{B}_\delta (\bar{X})$. Thus, we verify \eqref{intermed} by the (constant) selection $J:=I^*, w:=w_{I^*}(\bar{X}), t:=t_{I^*}(\bar{X})$  for each $X\in\mathbb{B}_\delta (\bar{X})$.

\vspace{0.1cm}
\noindent\underline{Case 2: $I^*\subsetneq I_0$}. First, we observe that we may assume $\#I^*=d$. Indeed, if $\#I^*<d$, then we may select some $j\in I_0\setminus I^*$. From $(w^*,t^*)\in\pm\{(w_{I^*}(\bar{X}),t_{I^*}(\bar{X})\}$ and by definition of $I_0$, we derive that $\langle w_{I^*}(\bar{X}),\bar{x}^{(j)}\rangle =t_{I^*}(\bar{X})$. Now, Lemma \ref{indexexpansion} yields that 
\[
(w_{I^*}(\bar{X}),t_{I^*}(\bar{X}))=(w_{I^*\cup\{j\}}(\bar{X}),t_{I^*\cup\{j\}}(\bar{X})).
\]
Therefore, we may have replaced the index set $I^*\in\Phi$ realizing the maximum in \eqref{formula_lambda} from the very beginning by the larger index set
$I^*\cup\{j\}\in\Phi$ (with $\#(I^*\cup\{j\})=\#I^*+1\leq d$) realizing the same value in \eqref{formula_lambda}. Calling this larger index set $I^*$ again,
we may proceed by adding further indices from $I_0$ to $I^*$ until $I^*=I_0$, in which case we are back to the situation we already dealt with above, or until $\#I^*=d$, which will be the setting we follow next.

\vspace{0.1cm}
\noindent\underline{Case 2.1: $(w^*,t^*)=(w_{I^*}(\bar{X}),t_{I^*}(\bar{X}))$}. The definition of $I_0$ and Proposition \ref{normed} yield that
\[
\langle w^*,\bar{x}^{(i)}\rangle =t^*=t_{I^*}(\bar{X})>0\quad\forall i\in I_0.
\]
Consequently, given an arbitrary $\varepsilon >0$, we may choose $\delta >0$ such that for all $X\in\mathbb{B}_\delta (\bar{X})\subseteq\mathcal{O}$,
\begin{eqnarray}
&\langle w^*,x^{(i)}\rangle\geq t^*/2>0\quad\forall i\in I_0\quad\mbox{and}\quad\mu^{Cap}(t_I(X))<\mu^{Cap}(t_I(\bar{X}))+\varepsilon\quad\forall I\in\Phi,&\label{continuity1}\\[1ex] %\nonumber\\
&
\left.
\begin{array}{l}
\langle w_{I}(\bar{X}),x^{(i)}\rangle >t_I(\bar{X})\Rightarrow \langle w_{I}(X),x^{(i)}\rangle >t_I(X)\\
\langle w_{I}(\bar{X}),x^{(i)}\rangle <t_I(\bar{X})\Rightarrow \langle w_{I}(X),x^{(i)}\rangle <t_I(X)
\end{array}
\right\}
\quad\forall i\in\{1,\ldots ,N\}\,\,\forall I\in\Phi ,&\label{continuity2}
\end{eqnarray}
where the continuity of $\mu^{Cap}$ and of the $w_I,t_I$ has been exploited. In order to verify \eqref{intermed}, we fix an arbitrary $X\in\mathbb{B}_\delta (\bar{X})$ and find $J,w,t$ as required there. To this aim, denote by $P$ the convex hull of the point set $\{x^{(i)}\}_{i\in I_0}$. 

\vspace{0.1cm}
\noindent\underline{Case 2.1.a): ${\rm int}\,P\neq\emptyset$}. Clearly, $0\notin P$ due to the first relation of \eqref{continuity1}. It is well-known from the theory of polyhedra (see, e.g., \cite[Theorem 2.15 (7)]{ziegler}), that there exists a representation
\begin{equation}\label{hrep}
P=\{x\in\mathbb{R}^d\mid\langle v_k,x\rangle\geq\tau_k\quad (k=1,\ldots , m)\}\quad (v_k\in\mathbb{S}^{d-1}, \tau_k\in\mathbb{R}),
\end{equation}
such that for $k'=1,\ldots ,m$ each set 
\[
F_{k'}:=\{x\in\mathbb{R}^d\mid\langle v_{k'},x\rangle =\tau_{k'},\,\,\langle v_k,x\rangle\geq\tau_k\quad (k=1,\ldots , m, k\neq k')\}
\]
is a facet of $P$. 
There exists some $k_0\in\{1,\ldots ,m\}$ such that $\tau_{k_0}>0$ because otherwise the contradiction $0\in P$ would result. As a facet of a bounded polyhedron $P\subseteq\mathbb{R}^d$ with ${\rm int}\,P\neq\emptyset$, $F_{k_0}$ must contain at least $d$ vertices of $P$. Since the vertices of $P$ are contained in the set $\{x^{(i)}\}_{i\in I_0}$, there exists a subset $J
\subseteq I_0$ with $\#J=d$ and $\{x^{(i)}\}_{i\in J}\subseteq F_{k_0}$. Hence, by definition of $F_{k_0}$, $X_J^Tv_{k_0}=\tau_{k_0}{\bf 1}$. 
By Proposition \eqref{normed}, $X_J^Tw_J(X)=t_J(X)\bf 1$, with $X_J^T$ being a regular $(d,d)$- matrix by genericity of $X$. Then,
\[
v_{k_0}=\tau_{k_0}\left(X_J^T\right)^{-1}{\bf 1},\quad w_J(X)=t_J(X)\left(X_J^T\right)^{-1}\bf 1.
\]
Since $\|v_{k_0}\|=\|w_J(X)\|=1$ and $\tau_{k_0},t_J(X)>0$, (see Proposition \ref{normed}), it follows that $(w_J(X),t_J(X))=(v_{k_0},\tau_{k_0})$. Now, \eqref{hrep} yields that
\begin{equation}\label{esti0}
\langle w_J(X),x^{(i)}\rangle =\langle v_{k_0},x^{(i)}\rangle\geq\tau_{k_0}=t_J(X)\quad\forall i\in I_0.
\end{equation}
With $\bar{X}_J^Tw_J(\bar{X})=t_J(\bar{X})$ (by Proposition \ref{normed})
and $\bar{X}_J^Tw^*=t^*$ (by $J\subseteq I_0$), the same reasoning as before yields that $(w_J(\bar{X}),t_J(\bar{X}))=(w^*,t^*)$. After having fixed $J\in\Phi$, we also fix $(w,t):=(w_J(X),t_J(X))$ as required in \eqref{intermed}. Then, by \eqref{esti0}, \eqref{continuity2} and by the definitions of $I_0, I_1$, one gets that
\begin{eqnarray}
\mu^{Emp}(X,w,t)&=&\mu^{Emp}(X,w_J(X),t_J(X))=N^{-1}\#\{i\in\{1,\ldots ,N\}\mid\langle w_J(X),x^{(i)}\rangle\geq t_J(X)\}\nonumber\\
&=&N^{-1}(\#I_0+\#I_1) \,=\,\mu^{Emp}(X,w^*,t^*),\label{muempest}
\end{eqnarray}
which is the first desired relation in \eqref{intermed}. The second one follows immediately from the second relation in \eqref{continuity1}. 

\vspace{0.1cm}
\noindent\underline{Case 2.1.b): ${\rm int}\,P=\emptyset$}. Then, $P$ as a polytope must be contained in some hyperplane $H$:
\[
P\subseteq H:=\{x\in\mathbb{R}^d\mid\langle\hat{w},x\rangle=\hat{t}\}\quad(\hat{w}\in\mathbb{S}^{(d-1)}, \hat{t}\in\mathbb{R}).
\]
We may assume that $\hat{t}\geq 0$. In particular,
$\langle\hat{w},x^{(i)}\rangle=\hat{t}$ for all $i\in I_0$,
or, $X_{I_0}^T\hat{w}=\hat{t}\bf{1}$, for short. Since also $X_{I^*}^Tw_{I^*}(X)=t_{I^*}(X)\bf 1$ (by Proposition \ref{normed}), and recalling that $\#I^*=d$ the same reasoning as above \eqref{esti0} yields that $(w_{I^*}(X),t_{I^*}(X))=(\hat{w},\hat{t})$. This implies that the choice $J:=I^*$ satisfies \eqref{esti0} (actually as an equation) so that in view of $(w^*,t^*)=(w_{I^*}(\bar{X}),t_{I^*}(\bar{X}))$ we may repeat the reasoning after \eqref{esti0} and \eqref{muempest} in order to derive the two relations in \eqref{intermed} in that alternative case too.

\vspace{0.1cm}
\noindent\underline{Case 2.2): $(w^*,t^*)=(-w_{I^*}(\bar{X}),-t_{I^*}(\bar{X}))$}. We observe that, in case of $\tau_k\geq 0$ for all $k=1,\ldots ,m$, $P$ must he unbounded according to \eqref{hrep} because from $x^{(1)}\in P$ it would follow that 
\[
\langle v_k,\lambda x^{(1)}\rangle =\lambda\langle v_k, x^{(1)}\rangle\geq\lambda \tau_k\geq \tau_k\quad\forall\lambda\geq 1\,\,\forall k=1,\ldots , m.
\]
This entails that $\lambda x^{(1)}\in P$ for all $\lambda\geq 1$, whence $P$ would be unbounded because $x^{(1)}\neq 0$ thanks to the genericity of $X$. However, $P$ is bounded as a convex combination of finitely many points. Therefore, there exists some $k_1\in\{1,\ldots ,m\}$ with $\tau_{k_1}<0$. Now, we can repeat exactly the argumentation from the first case above (referring to $\tau_{k_0}>0$), just with reversed signs.
\end{proof}

\bibliography{references}

\end{document}